\documentclass[11pt]{amsart}
\usepackage[pdftex]{hyperref}
\usepackage{amsfonts, amsthm, amsmath, amssymb}
\usepackage[all]{xy}
\usepackage[top=1in,bottom=1in,left=1in,right=1in]{geometry}

\def\C{\mathbb{C}}
\def\HH{\mathfrak{H}}
\def\M{\mathcal{M}}
\def\MM{\mathbb{M}}
\def\N{\mathbb{N}}

\def\Q{\mathbb{Q}}
\def\R{\mathbb{R}}
\def\Z{\mathbb{Z}}

\newcommand{\Aut}{\operatorname{Aut}}

\newcommand{\Hom}{\operatorname{Hom}}

\newtheorem{thm}{Theorem}[section]
\newtheorem*{thmm}{Theorem}
\newtheorem*{conj}{Conjecture}
\newtheorem{prop}[thm]{Proposition}
\newtheorem{lem}[thm]{Lemma}
\newtheorem{cor}[thm]{Corollary}

\theoremstyle{definition}
\newtheorem{defn}[thm]{Definition}

\theoremstyle{remark}

\begin{document}

\title{Generalized moonshine II: Borcherds products}
\author{Scott Carnahan}

\begin{abstract}
The goal of this paper is to construct infinite dimensional Lie algebras by using infinite product identities, and to use these Lie algebras to reduce the generalized moonshine conjecture to a pair of hypotheses about group actions on vertex algebras and Lie algebras.  We expect the Lie algebras that we construct to manifest as algebras of physical states in an orbifold conformal field theory (yet to be fully constructed) with symmetries given by the monster simple group.

We introduce vector-valued modular functions attached to families of modular functions of different levels, and we prove infinite product identities for a distinguished class of automorphic functions on a product of two half-planes.  We recast this result using the Borcherds-Harvey-Moore singular theta-lift, and show that the vector-valued functions attached to completely replicable modular functions with integer coefficients lift to automorphic functions with infinite product expansions at all cusps.  

For each element of the monster simple group, we construct an infinite dimensional Lie algebra, such that its denominator formula is an infinite product expansion of the automorphic function arising from that element's McKay-Thompson series.  These Lie algebras have the unusual property that their simple roots and all root multiplicities are known.  We show that under certain hypotheses, characters of groups acting on these Lie algebras form functions on the upper half-plane that are either constant or invariant under a genus zero congruence group.
\end{abstract}

\maketitle

\tableofcontents

\section*{Introduction}

In this paper, we are interested in automorphic functions derived from modular functions, infinite dimensional Lie algebras, and the generalized moonshine conjecture.  We use the machinery of automorphic infinite products to construct Lie algebras, and apply distinguished structural features of the Lie algebras to reduce the conjecture to some precisely stated hypotheses motivated by conformal field theory.

Let $\HH$ denote the complex upper half-plane, and let $j: \HH \to \C$ be Klein's modular invariant.  It is characterized as the unique holomorphic function on $\HH$ that is invariant under the action of $SL_2(\Z)$ by M\"obius transformations, with a simple pole of residue one at infinity, and a zero at a third root of unity.  The function $j$ is invariant under translation by one, so it admits a Fourier expansion $j(\tau) = q^{-1} + 744 + 196884q + \dots$, where $q = e^{2\pi i \tau}$, $\tau \in \HH$.  We shall consider an alternative normalization by defining $J(\tau) = j(\tau) - 744$.  This function is also holomorphic and invariant under $SL_2(\Z)$, and it is the unique such function whose Fourier expansion $\sum_n c(n) q^n$ has the form $q^{-1} + O(q)$, i.e., $c(n) = 0$ for $n<-1$, $c(-1) = 1$, and $c(0) = 0$.  During the 1980s, Koike, Norton, and Zagier independently proved the following remarkable formula:

\[ J(\sigma) - J(\tau) = p^{-1} \prod_{m > 0, n \in\Z} (1-p^m q^n)^{c(mn)} \qquad (p = e^{2 \pi i \sigma}, |p|<e^{-2\pi}, |q|<e^{-2\pi}). \]

One of the most interesting qualities of this identity is that the left side of the equation is a difference of two single-variable power series, while there is no obvious reason for the same to be true for the infinite product on the right side.  The cancellation of all mixed terms $p^iq^j, ij \neq 0$ implies infinitely many equations relating the coefficients $c(n)$.  For example, the vanishing of the $pq^2$ term on the right side implies $c(4) = c(3) + \binom{c(1)}{2}$, i.e., $20245856256 = 864299970 + \binom{196884}{2}$.

The product identity plays an essential role in monstrous moonshine, because it is the Weyl-Kac-Borcherds denominator formula for the monster Lie algebra, which is a rank two generalized Kac-Moody algebra that admits a homogeneous action of the monster simple group $\MM$.  The left side of the identity describes the action of the Weyl group on the simple roots together with a Weyl vector, and the exponents $c(mn)$ on the right side give the multiplicities of the positive roots, so the denominator formula can be employed to present the Lie algebra in terms of generators and relations.  In his proof of the monstrous moonshine conjecture \cite{B92}, Borcherds constructed this Lie algebra in two ways: the first construction, via generators and relations, yielded the denominator formula, and the second construction, by applying a certain functor to an object with a known monster action, endowed the Lie algebra with an action of the monster.  The object in question is the monster vertex algebra $V^\natural \cong \bigoplus_{n \geq -1} V_n$, constructed by Frenkel, Lepowsky, and Meurman, and the construction of the monster action on this structure requires most of a rather large book  \cite{FLM88} for building the required machinery.  To show that the two Lie algebras were isomorphic, Borcherds identified the root spaces of the Lie algebra with homogeneous spaces in $V^\natural$ to show that their root multiplicities matched.  Using the twisted denominator formulas arising from this action, together with facts about $V^\natural$ established in \cite{FLM88}, he showed that the characters of automorphisms of $V^\natural$ are equal to the genus zero modular functions listed in \cite{CN79}, and this proved the monstrous moonshine conjecture for the monster vertex algebra.

Even without a priori knowledge of the monster Lie algebra, one might expect that the Koike-Norton-Zagier formula is a sort of decategorification of a combinatorial or representation-theoretic object, since the coefficients of $J$ (and therefore exponents in the identity) are nonnegative integers.  Indeed, similar infinite product formulas appeared in the representation theory of Kac-Moody algebras well before moonshine was discovered.  However, the monster action makes this case a particularly rich example: for example, the identity $c(4) = c(3) + \binom{c(1)}{2}$ not only indicates the dimensions of homogeneous spaces in the Lie algebra, but gets promoted to an isomorphism $V_4 \cong V_3 \oplus \bigwedge^2 V_1$ of representations of the monster.

One might ask if this phenomenon is singular, or if there are similar identities, perhaps with modular forms of nonzero weight or level greater than one, that carry additional representation-theoretic information.  The answer seems to be positive, but limited, in the sense that there may be only finitely many interesting cases.  One can employ work of Borcherds, Harvey, and Moore to produce automorphic infinite products, and many of these are the Weyl denominators for Lie algebras that may appear in nature as algebras of BPS states or infinitesimal symmetries of a string theory \cite{HM95}.  However, among the known examples, there is rarely a natural way to endow the root spaces with group actions.  The known exceptions all have a basis in conformal field theory, and they include the fake monster Lie algebra with an action of automorphisms of a central extension of the Leech lattice \cite{B92}, some supersymmetric variants such as \cite{S00}, and a Lie algebra with baby monster action \cite{H03} constructed to solve a special case of Norton's generalized moonshine conjecture \cite{N87}.  This last construction is of principal interest in this paper.

The generalized moonshine conjecture arose from numerical experimentation, but is indicative of some vast structure:
\begin{conj}
There exists a map that attaches to each element $g$ of the monster, an infinite dimensional graded representation $V^\natural(g)$ of a central extension of the centralizer $C_\MM(g)$, satisfying the following properties:
\begin{enumerate}
\item For each $h \in C_\MM(g)$ and any lift $\tilde{h}$ to the chosen central extension, the graded character 
\[ Z(g,\tilde{h},\tau) = \sum_{n\in \Q} \operatorname{Tr}(\tilde{h}|V^\natural(g)_n) e^{2 \pi i (n-1) \tau} \]
is either constant or the $q$-expansion of a congruence Hauptmodul, i.e., a modular function on $\HH$ that is invariant under a group $\Gamma_{g,h}$, and generates the function field of the quotient $\HH/\Gamma_{g,h}$, where $\Gamma_{g,h} \supset \Gamma(N)$ for some $N >0$.
\item If we let $\mathcal{Z}(g,h,\tau)$ denote the set of functions $Z(g,\tilde{h},\tau)$ as $\tilde{h}$ ranges over lifts of $h$, then the functions in this set are constant multiples of each other, and the set of these functions is invariant under simultaneous conjugation of the pair $(g,h)$ in $\MM$.
\item For any $\binom{ab}{cd} \in SL_2(\Z)$, we have the equality of sets $\mathcal{Z}(g,h,\frac{a\tau+b}{c\tau+d}) = \mathcal{Z}(g^ah^c,g^bh^d,\tau)$.
\item $\mathcal{Z}(g,h,\tau)$ contains $J(\tau) = q^{-1} + 196884q + \dots$ if and only if $g=h=1 \in \MM$.
\end{enumerate}
\end{conj}
The assignment $(g,h) \mapsto \mathcal{Z}(g,h,\tau)$ of pairs of commuting elements of the monster to sets of functions on $\HH$ has been called a generalized character, since it appears to be related to the characters that show up in complex oriented cohomology theories \cite{HKR00}.  We use the notation $V^\natural$ here, even though the monster vertex algebra had not been constructed when the conjecture was posed, because in hindsight, the conjecture is fundamentally about the object constructed in \cite{FLM88}.

This conjecture is still open, but in some special cases, candidate representations have been explicitly constructed.  In those and some additional cases computed by Norton and Queen, it is known what $Z(g,h,\tau)$ must be, assuming the conjecture is true.  Furthermore, shortly after the conjecture was published, Dixon, Ginsparg, and Harvey \cite{DGH88} suggested that the representations of centralizers had a physical interpretation as twisted Hilbert spaces for a conformal field theory with monster symmetry.  While the conformal field theory in question has yet to be constructed in full, the physical interpretation has a rephrasing in terms of known algebraic objects: Using the notion of twisted module of a vertex algebra introduced in \cite{FLM88} for the construction of $V^\natural$, one replaces the twisted Hilbert spaces mentioned before with irreducible twisted modules of the monster vertex algebra $V^\natural$, so elements of $Z(g,h,\tau)$ are interpreted as graded characters of lifts of $h$ acting on an irreducible $g$-twisted module of $V^\natural$.

With the interpretation of the conjecture in terms of twisted modules in hand, much of the effort in proving it has focused on studying the twisted modules, and under the assumption that this interpretation is correct, candidate functions have been identified as characters in the following cases:
\begin{itemize}
\item Borcherds's proof of the original moonshine conjecture implies that $\mathcal{Z}(1,h,\tau)$ are congruence Hauptmoduln for all $h \in \MM$.  In this case, the twisted module is $V^\natural$, where the twisting is by the identity element.
\item Dong, Li, and Mason proved a general result about existence and uniqueness of irreducible twisted modules.  They then showed, by reducing to Borcherds's theorem, that, if $g$ and $h$ generate a cyclic group, then the resulting elements of $Z(g,h,\tau)$ are congruence Hauptmoduln that obey the $SL_2(\Z)$ compatibility \cite{DLM00}.
\item Hoehn showed that if $g$ lies in conjugacy class 2A, then the elements of $\mathcal{Z}(g,h,\tau)$ are congruence Hauptmoduln for all $h$ in the centralizer of $g$ \cite{H03}.
\end{itemize}
Hoehn's proof employs a nontrivially twisted module, but the overall structure of the proof is quite similar to Borcherds's proof of the moonshine conjecture, involving the construction of a generalized vertex operator algebra along the lines of \cite{FLM88} and the construction of a large Lie algebra with a finite group action by the use of a functor.

Based on Hoehn's success, it is reasonable to hope that additional cases of the generalized moonshine conjecture can be attacked with the strategy of constructing a Lie algebra with known root multiplicities (arising from a well-behaved product formula) and an action of a large finite group, and then evaluating twisted denominator formulas.  In a previous paper \cite{C10}, we showed that when the group and Lie algebra satisfied certain conditions, then the characters were automatically genus zero.  In particular, this allows us to circumvent the explicit computations at the end of \cite{H03}.  More importantly, this result makes it possible to analyze characters without classifying all pairs of commuting elements, and allows us to work with groups whose character tables are unknown.  In this paper, we will construct for each $g \in \MM$, a Lie algebra $W_g$ whose denominator identity recovers an automorphic infinite product.  When $g$ is the identity, we get the monster Lie algebra, and when $g$ is in conjugacy class 2A, we get Hoehn's Lie algebra.  These Lie algebras have the property that we know both the simple roots and the multiplicities of all roots, and they are given as coefficients of special functions.  Outside the case of affine Kac-Moody algebras, this is only known for a few special cases of infinite dimensional Lie algebras, which includes some constructions by Borcherds \cite{B98}, Gritsenko and Nikulin \cite{GN97}, \cite{GN98}, Hoehn \cite{H03}, and Scheithauer \cite{S08}, \cite{S04}, \cite{S09} (this list may be missing some contributions, and the author apologizes for those left out).  The positive subalgebra $E_g$ yields characters $Z(g,h,\tau)$ that behave as Norton conjectured, assuming some hypotheses we shall explain in detail.   The most important of these hypotheses is the existence of a suitable group action, and we intend to construct such an action later in this series using vertex algebra manipulations.

\section{Main results}

\begin{thmm} (Theorem \ref{thm:prods} and subsequent corollary)

If $f$ is a completely replicable modular function of level $N$ with integer coefficients, then there exists a vector-valued modular function $F$ canonically attached to $f$, such that the automorphic function $f(\sigma) - f(\tau)$ on $\HH \times \HH$ has an infinite product expansion at all cusps with exponents described by coefficients of $F$.  At the cusp $(i\infty,0)$, this product is:
\[ p^{-1} \prod_{m > 0, n \in \frac{1}{N}\Z} (1 - p^m q^n)^{c_{m,Nn}(mn)}, \]
where $p = e^{2 \pi i \sigma}$, $q = e^{2 \pi i \tau}$, and $c_{m,Nn}(mn)$ is the $q^{mn}$ coefficient of the $(m,Nn)$ component of $F$.
\end{thmm}

In this paper, we will adopt the convention that the McKay-Thompson series $T_g$ of an element $g$ in the monster is the graded character $Tr(gq^{L_0-1}|V^\natural)$ of $g$ acting on the monster vertex operator algebra $V^\natural$.  These series were only given conceptual grounding as characters of an object when $V^\natural$ was constructed in \cite{FLM88}, proving the McKay-Thompson conjecture asserting the existence of a natural monster representation with graded dimension equal to $J$.

\begin{thmm} (Theorem \ref{thm:gkm})
For each element $g$ of the monster simple group, let $T_g(\tau)$ be its McKay-Thompson series.  There exists a unique generalized Kac-Moody Lie algebra whose denominator formula is the product expansion of $T_g(\sigma) - T_g(\tau)$ at the cusp $(i\infty,0)$.
\end{thmm}

The following result requires unproven hypotheses that are introduced near the end of the paper.

\begin{thmm} (Theorem \ref{thm:condmoon})
Suppose hypothesis $C_g$ holds for all $g \in \MM$ and hypothesis $D$ holds for $V^\natural$ and all commuting pairs $(g,h)$.  Then the generalized moonshine conjecture is true (where the $SL_2(\Z)$-compatibility is only given up to a nonzero constant).
\end{thmm}

\section{Summary}

In section 3, we construct automorphic functions on the orthogonal group $O_{2,2}$ with product formulas in two different ways.  We begin by directly proving the formulas using equivariant Hecke operators.  After this, we show that many of these products arise from a Borcherds-Harvey-Moore regularized theta-lift.  The automorphic functions have a particularly degenerate form, in the sense that they analytically continue to functions on $\HH \times \HH$ whose invariance group produces the affine plane (minus some cusp lines) as a quotient space; under appropriate coordinate choices, the functions descend to have the form $x-y$.  This degenerate form is essential to our study, because it allows us to identify the simple roots of the Lie algebras in the following section.

In section 4, we present an application of this theory to generalized moonshine.  We show that the product formulas arising from McKay-Thompson series are the denominator formulas of generalized Kac-Moody algebras.  The main obstruction is the presence of exponents that cannot be dimensions of vector spaces, such as negative integers.  We show that under certain assumptions concerning the size of irreducible twisted modules of the monster vertex algebra, the root spaces of the Lie algebras can be identified with eigenspaces in the twisted modules.  Under an additional assumption concerning group actions on these Lie algebras, we show that the characters of the actions on the twisted modules are weakly Hecke-monic, and reduce the generalized moonshine conjecture to a conjecture concerning spaces of genus one functions.

We conclude in section 5 with some open problems.

\subsubsection*{Acknowledgements}
The author wishes to thank Richard Borcherds, Gerald Hoehn, and Jacob Lurie.  Borcherds suggested generalized moonshine as a dissertation project, and gave much useful advice and perspective.  Hoehn offered many helpful comments on an earlier version of this paper.  Lurie suggested that equivariant Hecke operators could be useful for understanding exponential elliptic cohomology operations, and this led to their use in product formulas.  This research was partly supported by NSF grant DMS-0354321.

\section{Automorphic infinite products}

The theory of automorphic infinite products originates in early work of Euler and Jacobi, but the representation-theoretic significance of the Fourier expansions of functions started with the independent discovery by Kac and Moody that the denominators of affine Kac-Moody algebras give rise to Macdonald identities (see e.g., \cite{K90}).  Borcherds expanded this work to generalized Kac-Moody algebras, and gave several examples, including the fake monster Lie algebra and monstrous Lie superalgebras.  Later work \cite{B95b} gave a coherent recipe for generating product formulas for automorphic forms on $O_{2,n-2}$ whose exponents are coefficients of modular forms.  Harvey and Moore reinterpreted these products using a regularized theta-lift \cite{HM95}, and for certain cases with integer multiplicities, they proposed a physical interpretation of the Lie algebras as infinitesimal symmetries of a heterotic string theory.  Borcherds then generalized their construction to work at higher level, using vector-valued modular forms $F$ that transform according to Weil's metaplectic representation (for more explanation of these ideas, see the introduction to \cite{B98} and references therein).

In this section, we construct automorphic functions on $\HH \times \HH$ with distinguished infinite product expansions at cusps.  We employ Borcherds's toolkit of vector-valued functions, but we first find product formulas using equivariant Hecke operators together with an analytic continuation argument.  After that, we introduce the Borcherds-Harvey-Moore theta-lift, and show that in many cases, the products we found come from this more general theory.

The regularized theta-lift has logarithmic singularities along certain quadratic divisors that specify linear fractional relations between the coordinates, and by exponentiating, we get a highly invariant function $\Psi$ with zeroes, poles, and possibly branches on the same loci.  Furthermore, for each vector-valued function $F$, the exponentiated lift $\Psi$ has an infinite product expansion at each cusp, with exponents given by coefficients of $F$.  We show that each McKay-Thompson series $T_g(\tau)$ from monstrous moonshine can be assembled into a vector-valued function whose lift is equal to $T_g(\sigma)-T_g(\tau)$.  This gives another way to identify the product expansion of $\Psi$ at $(i\infty,0)$ with $T_g(\sigma)-T_g(-1/\tau)$, and this will yield the Lie algebras we want in the next section.

The product expansions of $T_g(\sigma) - T_g(-1/\tau)$ were given in the author's Ph.D. dissertation, using an exhaustive enumeration of poles of all McKay-Thompson series.  We do not reproduce the argument here, because the current arguments using Hecke operators are more conceptual and easier to follow.

\subsection{Vector-valued functions} \label{sec:vector}

Let $M$ be a free abelian group of rank $4$, and let $Q: M \to \Z$ be a quadratic form.  We let $(u,v) = Q(u+v)-Q(u)-Q(v)$ be the corresponding even bilinear form, and we assume it is nondegenerate with signature $(2,2)$.  We write $e(x) = e^{2 \pi i x}$ for the normalized exponential, and $M^\vee := \Hom(M,\Z)$ for the dual lattice.  The bilinear form induces an inclusion of $M$ into $M^\vee$ as a finite index subgroup, and $M^\vee \! /M$ inherits a $\Q/\Z$-valued quadratic form which we will call $Q$.

\begin{defn} A vector-valued modular function of type $\rho_M$ is a meromorphic function $F: \HH \to \C[M^\vee \! /M]$, written as $F(\tau) = \sum_{\gamma \in M^\vee \! /M} F_\gamma(\tau) e_\gamma$, satisfying
\[ \begin{aligned}
F_\gamma(\tau + 1) &= e(Q(\gamma)) F_\gamma \\
F_\gamma(-1/\tau) &= |M^\vee \! /M|^{-1/2} \sum_{\delta \in M^\vee \! /M} e(-(\gamma,\delta)) F_\delta(\tau)
\end{aligned} \]
\end{defn}

These functions produce finite-dimensional representations of $SL_2(\Z)$, known as the Weil representations associated to $M$.  In the following paragraphs, we describe a technique for constructing vector-valued functions suggested by Borcherds.  Essentially, one combines modular functions of various levels and applies lemma 2.6 in \cite{B98}.

We define the lattice $I\!I_{1,1}$ to be the group $\Z \times \Z$, equipped with the quadratic form $Q(a,b) = ab$.  The associated bilinear form is even and unimodular.  For a positive integer $N >0$, we define $I\!I_{1,1}(N)$ to be the group $\Z \times \Z$ with quadratic form $Q(a,b) = Nab$.  If we identify $I\!I_{1,1}(N)$ with the sublattice of $I\!I_{1,1}$ whose underlying group is $\Z \times N\Z$, then the dual lattice is identified with $\frac{1}{N}\Z \times \Z$, and the quotient is isomorphic to $\Z/N\Z \times \Z/N\Z$ with quadratic form $Q(a,b) = ab/N \in \Q/\Z$.

\begin{lem} \label{lem:basic}
The group $\Z/N\Z \times \Z/N\Z$ decomposes into $\sigma_0(N)$ orbits under the right vector action of $SL_2(\Z)$ following reduction modulo $N$, i.e., where $(i, j) \cdot \binom{ab}{cd} = (ai+cj, bi+dj)$, and $\sigma_0(N)$ is the number of positive integers dividing $N$.  There is an explicit bijection between orbits and positive integer divisors of $N$, given by attaching to each $m|N$, the orbit that contains the element $(0,m)$.
\end{lem}
\begin{proof}
For any $(i,j) \in \Z \times \Z$, let $m = (i,j,N)$.  Then there exists a (not necessarily unique) matrix $A$ of the following form:
\[ A = \begin{pmatrix} \ast & \ast \\ i/m & j/m \end{pmatrix} \in SL_2(\Z/(N/m)\Z). \]
If we replace the integers $i,j,m$ with their residue classes mod $N$, we have $(0,m) \cdot A \equiv (i,j)$.  Since the reduction map $SL_2(\Z) \to SL_2(\Z/\frac{N}{m}\Z)$ is surjective, the orbit of $(0,m)$ under $SL_2(\Z)$ is equal to the set of all $(k,\ell) \in \Z/N\Z \times \Z/N\Z$ satisfying $(k,\ell,N) = m$.  This yields the desired bijection.
\end{proof}

For each $m|N$, let $f_{(m)}: \HH \to \C$ be a modular function invariant under $\Gamma_1(N/m)$.  We construct a $\C[\Z/N\Z \times \Z/N\Z]$-valued function $\hat{F}$ on $\HH$ by setting $\hat{F}_{i,j}(\tau) = f_{((i,j,N))}(A \tau)$, where $A$ is as given in the proof of Lemma \ref{lem:basic}.  In particular, $\hat{F}_{0,m} = f_{(m)}$, and for any $A \in SL_2(\Z)$, we have $\hat{F}_{(0,m) \cdot A} = f_{(m)}(A\tau)$.  This definition yields the following invariance relation: $\hat{F}_{i,j}(\frac{a\tau+b}{c\tau+d}) = \hat{F}_{ai+cj,bi+dj}(\tau)$ for $\binom{ab}{cd} \in SL_2(\Z)$.  We can think of $\hat{F}$ as a way to encode all $SL_2(\Z)$ transformations of $f_{(m)}$ as $m$ ranges over divisors of $N$.  We note that this assignment yields a well-defined function, because the ambiguity in our choice of $A$ is given by left translation by elements of $\Gamma_1(N/m)$.

Let $M = I\! I_{1,1}(N) \times I\! I_{1,1}$ and fix the inclusion $M \subset I\! I_{1,1} \times I\! I_{1,1}$ as $\Z \times N\Z \times \Z \times \Z \subset \Z^4$ with quadratic form $Q(a,b,c,d) = ab+cd$.  Then we identify $M^\vee \! /M = (\frac{1}{N}\Z \times \Z)/(\Z \times N\Z)$ with $\Z/N\Z \times \Z/N\Z$ with a coordinate switch: $(a,b,0,0) + M \mapsto (b+N\Z, Na+N\Z)$, and define $F$ to be the Fourier transform of $\hat{F}$ in the second coordinate, i.e., along the rows of a matrix of functions:
\[ F_{i,k} = \frac{1}{N} \sum_{j \in \Z/N\Z} e(-jk/N) \hat{F}_{i,j} \]
\begin{lem} 
$F$ is a vector-valued function of type $\rho_M$.
\end{lem}
\begin{proof}
We need to check two conditions:
\[ \begin{aligned}
F_{i,k}(\tau + 1) &= \frac{1}{N} \sum_{j \in \Z/N\Z} e(-jk/N) \hat{F}_{i,j}(\tau + 1) \\
&= \frac{1}{N} \sum_{j \in \Z/N\Z} e(-jk/N) \hat{F}_{i,j+i}(\tau) \\
&= \frac{1}{N} e(ik/N) \sum_{j \in \Z/N\Z} e(-(j+i)k/N) \hat{F}_{i,j+i}(\tau) \\
&= e(ik/N)F_{i,k}(\tau)
\end{aligned} \]
so $F$ behaves correctly under $T: \tau \mapsto \tau+1$.
\[ \begin{aligned}
F_{i,k}(-1/\tau) &= \frac{1}{N} \sum_{j \in \Z/N\Z} e(-jk/N) \hat{F}_{i,j}(-1/\tau) \\
&= \frac{1}{N} \sum_{j \in \Z/N\Z} e(-jk/N) \hat{F}_{j,-i}(\tau) \\
&= \frac{1}{N} \sum_{j \in \Z/N\Z} e(-jk/N) \sum_{l \in \Z/N\Z} e(-il/N) F_{j,l}(\tau) \\
&= \frac{1}{N} \sum_{j, l \in \Z/N\Z} e(\frac{-jk-il}{N}) F_{j,l}(\tau)
\end{aligned} \]
so $F$ behaves correctly under $S: \tau \mapsto -1/\tau$.
\end{proof}

\begin{prop} \label{prop:vectisom} 
The Fourier transform described above is a linear isomorphism between the space of families $\{ f_{(m)} \}_{m|N}$ of modular functions with $f_{(m)}$ invariant under $\Gamma_1(N/m)$ and the space of vector-valued functions of type $\rho_M$.
\end{prop}
\begin{proof}
The map is clearly linear, and by the previous lemma, it takes families $\{ f_{(m)} \}_{m|N}$ to vector-valued functions of type $\rho_M$.  Since the Fourier transform is invertible, the kernel of this map is trivial.  It suffices to show that the map is surjective.

Let $F$ be a vector-valued function of type $\rho_M$.  We apply the inverse Fourier transform to get a $\C[\Z/N\Z \times \Z/N\Z]$-valued function $\hat{F}$ on $\HH$:
\[ \hat{F}_{i,j}(\tau) = \sum_{k \in \Z/N\Z} e(jk/N) F_{i,k}(\tau) \]
We wish to show that for each $m|N$, $\hat{F}_{0,m}(\tau)$ is invariant under $\Gamma_1(N/m)$ and that for $m = (i,j,N)$, $\hat{F}_{i,j}(\tau) = \hat{F}_{0,m}(\frac{a\tau+b}{c\tau + d})$ for all $a,b,c,d \in \Z$ satisfying $ad-bc = 1$, $c \equiv i/m$ (mod $N$), and $d \equiv j/m$ (mod $N$).  Since the two maps $\tau \mapsto \tau+1$ and $\tau \mapsto -1/\tau$ generate $SL_2(\Z)$, it suffices to show that $\hat{F}_{i,j}(\tau+1) = \hat{F}_{i,j+i}(\tau)$ and $\hat{F}_{i,j}(-1/\tau) = \hat{F}_{j,-i}(\tau)$.

\[ \begin{aligned}
\hat{F}_{i,j}(\tau+1) &= \sum_{k \in \Z/N\Z} e(jk/N) F_{i,k}(\tau + 1) \\
&= \sum_{k \in \Z/N\Z} e(jk/N) e(ik/N) F_{i,k}(\tau) \\
&= \sum_{k \in \Z/N\Z} e(k(j+i)/N) F_{i,k}(\tau) \\
&= \hat{F}_{i,j+i}(\tau)
\end{aligned} \]
so $\hat{F}$ behaves correctly under $T$.
\[ \begin{aligned}
\hat{F}_{i,j}(-1/\tau) &= \sum_{k \in \Z/N\Z} e(jk/N) F_{i,k}(-1/\tau) \\
&= \sum_{k \in \Z/N\Z} e(jk/N) \frac{1}{N} \sum_{m,n \in \Z/N\Z} e(\frac{-in-km}{N}) F_{m,n}(\tau) \\
&= \sum_{n \in \Z/N\Z} e(-in/N) \frac{1}{N} \sum_{m \in \Z/N\Z} F_{m,n}(\tau) \sum_{k \in \Z/N\Z} e(\frac{k(j-m)}{N}) \\
&= \sum_{n \in \Z/N\Z} e(-in/N) F_{j,n}(\tau) \\
&= \hat{F}_{j,-i}(\tau)
\end{aligned} \]
so $\hat{F}$ behaves correctly under $S$.

To conclude, we note that:
\begin{enumerate}
\item For all $m$, $\hat{F}_{0,m}$ is invariant under $T$, since $\hat{F}_{0,m} = \hat{F}_{0,m+0}$.
\item For all $m|N$, $\hat{F}_{0,m}$ is invariant under $\Gamma(N/m)$, since $\hat{F}_{0,m}(\frac{(aN/m + 1)\tau + (bN/m)}{(cN/m)\tau + (dN/m + 1)}) = \hat{F}_{0,m}(\tau)$ for all $a,b,c,d \in \Z$ such that $\begin{pmatrix} aN/m + 1 & bN/m \\ cN/m & dN/m + 1 \end{pmatrix}$ has determinant $1$.
\end{enumerate}
Since $\Gamma_1(N/m)$ is generated by $\Gamma(N/m)$ and $T$, $\hat{F}_{0,m}$ is invariant under $\Gamma_1(N/m)$.
\end{proof}

\begin{cor} \label{cor:restrict} 
A vector-valued function $F$ of type $\rho_M$ is uniquely determined by its restriction to components $F_{0,i}$ for $i \in \Z/N\Z$.  In particular, if these components are zero, then $F$ is zero.
\end{cor}
\begin{proof}
$F$ is uniquely determined by its inverse Fourier transform $\hat{F}$.  $\hat{F}$ is uniquely determined by its restriction to the components $\hat{F}_{0,m}$ for $0 \leq m < N$, and these components are uniquely determined by their Fourier transforms $F_{0,i}$.
\end{proof}

\subsubsection{Moduli interpretation}

Modular functions (or forms) can be viewed as functions (or sections of line bundles) on moduli spaces of structured elliptic curves.  We can think of the functions $\hat{F}$ in a similar way, except unlike the usual cases, this moduli space is not in general connected.  Since each $\hat{F}$ is determined by a set of functions $\{ f_{(m)} \}_{m|N}$ where $f_{(m)} : \HH \to \C$ is invariant under $\Gamma_1(N/m)$, $\hat{F}$ is a function on the disjoint union of modular curves $\coprod_{m|N} Y_1(N/m)$.  Each curve $Y_1(N/m)$ is the coarse space for the moduli problem classifying elliptic curves equipped with a point of order exactly $N/m$, i.e., as an analytic object, $Y_1(N/m)$  receives a universal map from the fibered category over complex analytic spaces whose objects are diagrams $\xymatrix{ E \ar[r] & S \ar@<0.5ex>[l]^e \ar@<-0.5ex>[l]_p }$ of complex analytic spaces, with $E \to S$ smooth and proper with one dimensional genus one fibers, and $e,p: S \to E$ sections, such that $p$ has order $N/m$ under the $S$-group law induced by $e$, and whose morphisms are fibered diagrams.  Using this interpretation, we can think of $\hat{F}$ as a function on the space of elliptic curves $E \overset{e}{\leftrightarrows} S$ equipped with a point $p: S \to E$, such that $[N]p = e$.

We have been viewing the component functions of $\hat{F}$ as functions on the complex upper half-plane, so it will be convenient to consider the above disjoint union of curves as the quotient of a disjoint union of half-planes by an action of $SL_2(\Z)$, in a way that specializes to the classical uniformization of $Y(1)$ when $N=1$.  We begin by recalling two geometric objects attached to an arbitrary finite group $G$, introduced in Section 2 of \cite{C10}:
\begin{enumerate}
\item Consider the category whose objects are diagrams $\xymatrix{ P \ar[r]^a & E \ar[r]^\pi & S \ar@/^/[ll]^{\tilde{e}} }$ of complex analytic spaces equipped with sheaf isomorphisms $\psi: R^1\pi_*\underline{\Z} \to \underline{\Z \times \Z}$, such that:
\begin{itemize}
\item $a: P \to E$ is a $G$-torsor,
\item $\pi: E \to S$ is a smooth proper map whose fibers are genus one curves,
\item $\tilde{e}: S \to P$ is a section of $\pi \circ a$,
\item The exterior square $\wedge^2 \psi: \wedge^2 R^1\pi_* \underline{\Z}(1) \overset{\cup}{\to} R^2 \pi_* \underline{\Z}(1) \to \underline{\Z}$ is the negative of the canonical isomorphism,
\end{itemize}
and whose morphisms are fibered diagrams, where the maps $P \to P'$ over $E \to E'$ are $G$-equivariant, equipped with the canonical isomorphisms of sheaves under pullback.  This stack is represented by the analytic space $\Hom(\Z \times \Z, G) \times \HH$.  By dualizing $\phi$, we find that the connected component attached to a fixed homomorphism $\phi: \Z \times \Z \to G$ parametrizes complex analytic families of elliptic curves $E$ with $G$-torsors $P$ and a preferred basis of fiberwise $H_1(E) \cong \Z \times \Z$, such that the monodromy of $P$ along a basis element in $H_1$ from the base point is described by applying $\phi$.  When $N=1$, we have $P=E$, and the above category reduces to the moduli description of the upper half-plane given in \cite{D69} and \cite{C?}.
\item Consider the category whose objects are diagrams $P \overset{a}{\to} E \overset{e}{\underset{\pi}{\leftrightarrows}} S$ of complex analytic spaces, such that:
\begin{itemize}
\item $a: P \to E$ is a $G$-torsor,
\item $\pi: E \to S$ is a smooth proper map whose fibers are genus one curves,
\item $e: S \to E$ is a section of $\pi$,
\end{itemize}
and whose morphisms are fibered diagrams, where the maps $P \to P'$ over $E \to E'$ are $G$-equivariant.  This is commonly called the complex analytic stack of elliptic curves with $G$-torsors, and we will denote it by $\M^G_{Ell}$.
\end{enumerate}
The second stack is the quotient of the first by a canonical action of $G \times SL_2(\Z)$, where elements of $G$ translate among lifts $\tilde{e}$ of $e$, and elements of $SL_2(\Z)$ transform the isomorphism $\psi$.  In terms of coordinates, the points of the first space $\Hom(\Z \times \Z, G) \times \HH$ are given by triples $(g,h,\tau)$, where $g,h \in G$ commute, and $\tau \in \HH$.  The action of $G$ identifies $(g,h, \tau)$ with the simultaneous conjugates $(kgk^{-1}, khk^{-1}, \tau)$ as $k$ ranges over elements of $G$, since changes of base point conjugate the monodromy of a torsor.  The action of $SL_2(\Z)$ identifies $(g^ah^c, g^bh^d, \tau)$ with $(g,h,\frac{a\tau+b}{c\tau+d})$ for all $\binom{ab}{cd} \in SL_2(\Z)$.  We can view a function on $\M^G_{Ell}$ as a function $f(g,h,\tau)$ on $\Hom(\Z \times \Z, G) \times \HH$ that is invariant under this action of $G \times SL_2(\Z)$.

We specialize this picture to the case $G = \Z/N\Z$.  Functions on $\M^{\Z/N\Z}_{Ell}$ are functions on $\Hom(\Z \times \Z, \Z/N\Z) \times \HH$ that are invariant under the action of $\Z/N\Z \times SL_2(\Z)$ described in the previous paragraph.  In contrast to the general case, the action of $\Z/N\Z$ is trivial, since it is given by conjugation in an abelian group.  We then have an identification of the coarse moduli space with a disjoint union of analytic quotients of the upper half plane, and by the analysis of $SL_2(\Z)$-orbits in $\Z/N\Z \times \Z/N\Z$ in Lemma \ref{lem:basic}, it is a disjoint union of quotients by $\Gamma_1(N/m)$ as $m$ ranges over divisors of $N$.  In particular, this identification with $\coprod Y_1(N/m)$ allows us to make an identification of function spaces.  We summarize the situation as follows:

\begin{lem} \label{lem:isoms}
For any positive integer $N$, we have linear isomorphisms between the spaces:
\begin{enumerate}
\item Families $\{ f_{(m)} \}_{m|N}$ of modular functions, where each $f_{(m)}$ is invariant under $\Gamma_1(N/m)$.
\item Functions $f$ on $\M^{\Z/N\Z}_{Ell}$.
\item Functions $\hat{F}_{i,j}(\tau)$ on $\Z/N\Z \times \Z/N\Z \times \HH$ that obey the relation $\hat{F}_{i,j}(\frac{a\tau+b}{c\tau+d}) = \hat{F}_{ai+cj,bi+dj}(\tau)$ for $\binom{ab}{cd} \in SL_2(\Z)$.
\item Vector-valued functions of type $\rho_M$ for $M = I\! I_{1,1}(N) \times I\! I_{1,1}$.  
\end{enumerate}
\end{lem}
\begin{proof}
The isomorphism $3 \leftrightarrow 1$ is given by restriction: $f_{(m)} = \hat{F}_{0,m}$.  The fact that this is an isomorphism is a consequence of Lemma \ref{lem:basic}.

The bijection $2 \leftrightarrow 3$ is due to the identification of the coarse moduli space of $\M^{\Z/N\Z}_{Ell}$ as an analytic quotient of $\Z/N\Z \times \Z/N\Z \times \HH$ by the action of $SL_2(\Z)$ given above.

The correspondence between $4 \leftrightarrow 1$ was proved in Proposition \ref{prop:vectisom}.
\end{proof}

\subsubsection{Rationality}

Let $f$ be a function on $\M_{Ell}^{\Z/N\Z}$, and let $F$ be the corresponding vector-valued form.  We give sufficient conditions for $F$ to have rational coefficients.

\begin{lem} \label{lem:mcgraw1} 
If the Fourier coefficients $c_{0,k}(n) \in \Q$ for all $k \in \Z/N\Z$, and $n \in \Q$, then all coefficients of $F$ are rational.  The analogous statement holds if we replace $\Q$ with any subring of $\C$ containing $\Q$.
\end{lem}
\begin{proof}
From the main theorem of \cite{M03}, the space of vector-valued functions of type $\rho_M$ has a basis consisting of functions all of whose Fourier expansions have only integer coefficients.  If we choose such a basis $\{ f^i \}$, then there exist rational numbers $a_i$ such that $F_{0,k} = \sum_i a_i f^i_{0,k}$ for all $k$.  By corollary \ref{cor:restrict}, $F$ is uniquely determined by its restriction to $F_{0,k}$, so $F = \sum_i a_i f^i$.

This argument works for any ring containing $\Q$, since we may allow the elements $a_i$ to lie in the larger ring.
\end{proof}

We define $\hat{c}_{i,j}(n)$ to be the $q^n$ coefficient of $f(g^i,g^j,\tau)$, so its $q$-expansion is $\sum \hat{c}_{i,j}(n) q^n$.  We define $c_{i,k}(n) = \frac{1}{N} \sum_{j\in \Z/N\Z} e(-jk/N) \hat{c}_{i,j}(n)$, so these are the coefficients of the vector-valued function $F$.

\begin{lem} \label{lem:mcgraw2} 
If $c_{i,k}(n) \in \Q$ for all $i,k \in \Z/N\Z$, $n<0$, and $c_{0,k}(0) \in \Q$ for all $k$, then all coefficients of $F$ are rational.  The analogous statement holds if we replace $\Q$ with any subring of $\C$ containing $\Q$.
\end{lem}
\begin{proof}
By the main theorem of \cite{M03}, there exists a basis $\{ f^i \}$ for the space of vector-valued functions of type $\rho_M$ with integer coefficients.  Let $V$ be the space of vector-valued functions of type $\rho_M$ such that the expansions of all components are regular, and the constant terms of the $(0,k)$ components are zero.  By our hypotheses, there exists a set $\{ a_i \}$ of rational numbers such that $F - \sum_i a_i f^i$ lies in $V$.  The inverse Fourier transform gives a bijection between $V$ and the space of families $\{ f_{(m)} \}_{m|N}$ of modular functions that are regular at all cusps, and whose expansions at infinity have vanishing constant term.  The regularity implies these modular functions are constants, and the vanishing of the constant term implies these functions are zero.  $V$ is therefore composed of the constant function zero, and $F = \sum_i a_i f^i$.

As in the previous lemma, this argument works for larger rings, since we can choose $a_i$ appropriately.
\end{proof}

\subsection{Hecke-monic modular functions} \label{sec:hecke-monic}

Recall the following definitions from \cite[Sections 1,2]{C10}.
\begin{enumerate}
\item Given a finite group $G$, the equivariant Hecke operator $nT_n$ is defined for functions on $\M_{Ell}^G$ by summing over pullbacks of $G$-torsors over degree $n$ isogenies.  When $G$ is cyclic of order $N$ and generated by $g$, we can write a formula for $nT_n$ in terms of functions on the complex upper half-plane:
\[ nT_nf(g^i,g^j,\tau) = \underset{0 \leq b < d}{\sum_{ad=n}} f(g^{di}, g^{aj-bi}, \frac{a\tau+b}{d}) \]
\item A function $f$ on $\M_{Ell}^{\Z/N\Z}$ is Hecke-monic if and only if for all $n > 0$ and any connected component $X \subset \M_{Ell}^{\Z/N\Z}$ there exists a monic polynomial $\Phi_{X,n} \in \mathbb{C}[x]$ of degree $n$ such that the restriction of $nT_nf$ to $X$ is equal to the restriction of $\Phi_{X,n}(f)$ to $X$.
\end{enumerate}

\begin{defn}
A function on $\M_{Ell}^{\Z/N\Z}$ is principally normalized if $f(1,g^i,\tau)$ has a $q$-expansion of the form $q^{-1} + O(q)$ for all $i \in \Z/N\Z$.
\end{defn}

\begin{lem}
If $f$ is a principally normalized Hecke-monic function on $\M_{Ell}^{\Z/N\Z}$, then for each $m|N$, $f(1,g^m,\tau)$ is a normalized Hauptmodul invariant under $\Gamma_0(N/m)$.
\end{lem}
\begin{proof}
Since $f$ is Hecke-monic, and $f(1,g^k,\tau)$ has $q$-expansion of the form $q^{-1} + O(q)$ for all $k$, Corollary 3.6 in \cite{C10} implies each $f(1,g^k,\tau)$ is a normalized Hauptmodul, i.e., each $f(1,g^k,\tau)$ has the $q$-expansion of the required form, is invariant under a genus zero group of M\"obius transformations, and generates the function field of the upper half-plane quotient.

Let $\binom{ab}{cd} \in \Gamma_0(N/m)$.  Then $f(1,g^m,\frac{a\tau+b}{c\tau+d}) = f(1^a g^{cm}, 1^b g^{dm}, \tau) = f(1,g^{dm},\tau)$.  We assumed $f$ was principally normalized, so the two functions $f(1,g^m,\frac{a\tau+b}{c\tau+d})$ and $f(1,g^m,\tau)$ both have $q$-expansions of the form $q^{-1}+O(q)$.  By the Hauptmodul property of $f(1,g^m,\tau)$, the expansions must be equal.
\end{proof}

\begin{lem} \label{lem:logp} 
Let $f$ be a principally normalized Hecke-monic function on $\M_{Ell}^{\Z/N\Z}$.  Then, \[ \log p(f(1,g,\sigma)-f(1,g,\tau)) = - \sum_{m>0} T_m f(1,g,\tau) p^m \]
in a neighborhood of $p=q=0$, where $p=e(\sigma)$ and $q=e(\tau)$.
\end{lem}
\begin{proof}
Following \cite{N84}, we define numbers $H_{m,n}$ by the bivarial transform (see the next section for details): the power series $q^{-n} + n\sum_{m} H_{m,n}q^m$ satisfies two properties:
\begin{enumerate}
\item it is $n$ times the coefficient of $p^n$ in $-\log p(f(1,g,\sigma)-f(1,g,\tau))$
\item it is the unique polynomial in $f(1,g,\tau)$ whose expansion has the form $q^{-n} + O(q)$.
\end{enumerate}
By \cite{C10} Lemma 5.2, $mT_m f(1,g,\tau)$ is also the unique polynomial in $f(1,g,\tau)$ whose expansion has the form $q^{-n} + O(q)$.  Therefore, the expansion of $\log p(f(1,g,\sigma)-f(1,g,\tau))$ as a series in $p$ has coefficients given by $-T_m f(1,g,\tau)$.

By the genus zero property, $f(1,g,\tau)$ is injective in a nontrivial neighborhood of zero in the $q$-disc, so $\log p(f(1,g,\sigma)-f(1,g,\tau))$ is the sum of $\log (1-pq^{-1})$ and a power series that converges absolutely in a nontrivial neighborhood of $p=q=0$ in the unit polydisc.  Therefore, the product also converges on the complement of the locus $pq=0$ in this neighborhood, and the formal equality extends to a nontrivial domain.
\end{proof}


\begin{lem} \label{lem:prod1} 
Let $f$ be a principally normalized Hecke-monic function on $\M_{Ell}^{\Z/N\Z}$, and let $F$ be the corresponding vector-valued form with coefficients $c_{i,k}(n)$, as defined just before Lemma \ref{lem:mcgraw2}.  Then $f(1,g,\sigma)-f(1,g,\tau)$ has an infinite product expansion of the form:
\[ p^{-1} \prod_{m>0,n \in \Z} \prod_{j \in \Z/N\Z} (1 - e(j/N)p^m q^n)^{c_{0,j}(mn)} \]
in a neighborhood of the cusp $(i\infty,i\infty)$ of $\HH \times \HH$.
\end{lem}
\begin{proof}
We multiply by $p$ and take the logarithm:
\[ \begin{aligned}
\sum_{m>0,n \in \Z} & \sum_{j \in \Z/N\Z} c_{0,j}(mn) \log(1 - e(j/N)p^m q^n) \\
&= \sum_{m>0,n \in \Z} \sum_{j \in \Z/N\Z} c_{0,j}(mn) \sum_{i=1}^\infty \frac{-1}{i} e(ij/N) p^{mi} q^{ni} \\
&= - \sum_{m>0,n \in \Z} \sum_{0<a|(m,n)} \frac{1}{a} \sum_{r \in \frac{1}{N}\Z/\Z} c_{0,Nr}(mn/a^2) e(ar)  p^m q^n \\
&= - \sum_{m>0} p^m \sum_{ad=m} \frac{1}{a} \sum_{n \in \Z} \hat{c}_{0,a}(dn) q^{an} \\
&= - \sum_{m>0} p^m \sum_{ad=m} \frac{1}{a} \sum_{0 \leq b < d} \frac{e(b/d)}{d} \sum_{n \in \Z} \hat{c}_{0,a}(n) q^{an/d} \\
&=  - \sum_{m>0} p^m \frac{1}{m} \underset{0 \leq b < d}{\sum_{ad=m}} \sum_{n \in \Z} \hat{c}_{0,a}(n)e(b/d) q^{an/d} \\
&= - \sum_{m>0} p^m \frac{1}{m} \underset{0 \leq b < d}{\sum_{ad=m}} \hat{F}_{0,a}(\frac{a\tau + b}{d}) \\
&= - \sum_{m>0} T_mf(1,g,\tau) p^m
\end{aligned} \]

By Lemma \ref{lem:logp}, this is equal to $\log p(f(1,g,\sigma)-f(1,g,\tau))$, and both converge on a nontrivial neighborhood of $p=q=0$ in the unit polydisc.  Therefore, the product is equal to $f(1,g,\sigma)-f(1,g,\tau)$ on a nontrivial neighborhood of $p=q=0$.
\end{proof}


We now wish to consider a product expansion of $f(1,g,\sigma)-f(1,g,\tau)$ in a neighborhood of the cusp $(i\infty,0)$.  Here, we take ``expansion of $f(1,g,\tau)$ at zero'' to mean the $q$-expansion of $f(1,g,-1/\tau)$.  Since $f(1,g,-1/\tau) = f(g,1,\tau)$, we are really seeking an infinite product expansion of $f(1,g,\sigma)-f(g,1,\tau)$ at $(\infty,\infty)$.

\begin{prop} \label{prop:prod2} 
Let $f$ be a principally normalized Hecke-monic function on $\M_{Ell}^{\Z/N\Z}$, and let $F$ be the corresponding vector-valued function, with coefficients $c_{i,k}(n)$ as in Lemma \ref{lem:prod1}.  Then in a neighborhood of infinity, $f(1,g,\sigma) - f(g,1,\tau)$ has a product expansion:
\[ p^{-1} \prod_{m>0,n \in \frac{1}{N}\Z} (1 - p^m q^n)^{c_{m,Nn}(mn)} \]
\end{prop}
\begin{proof}

We multiply the product by $p$ and take the logarithm to get:

\[ \begin{aligned}
\sum_{m>0,n \in \frac{1}{N}\Z} & c_{m,Nn}(mn) \sum_{i > 0} \frac{-1}{i} p^{mi}q^{ni} \\
&= - \sum_{m > 0, n \in \frac{1}{N}\Z} \sum_{0<a|(m,Nn)} \frac{1}{a} c_{m/a,Nn/a}(mn/a^2) p^m q^n \\
&= - \sum_{m > 0} \sum_{ad=m} \frac{1}{a}\sum_{n \in \frac{1}{N}\Z} c_{d,Nn}(dn) p^m q^{an} \\
&= - \sum_{m > 0} \sum_{ad=m} \frac{1}{a} \sum_{0 \leq b < d} \frac{1}{d} \sum_{n \in \frac{1}{N}\Z} \underset{n \in dr+\Z}{\sum_{r \in \frac{1}{N}\Z/\Z}} e(-br) c_{d,Nr}(n) e(br) q^{an/d} p^m \\
&= - \sum_{m > 0} \frac{1}{m} \underset{0 \leq b < d}{\sum_{ad=m}} \sum_{n \in \frac{1}{N}\Z} \underset{n \in dr+\Z}{\sum_{r \in \frac{1}{N}\Z/\Z}} e(-br) c_{d,Nr}(n) e\left(n\frac{a\tau + b}{d}\right) p^m \\
&= -\sum_{m > 0} \frac{1}{m} \underset{0 \leq b < d}{\sum_{ad=m}} \sum_{n \in \frac{1}{N}\Z} \hat{c}_{d,-b}(n) e\left(n\frac{a\tau + b}{d}\right) p^m \\
&= -\sum_{m > 0} \frac{1}{m} \underset{0 \leq b < d}{\sum_{ad=m}} f\left(g^d, g^{-b},\frac{a\tau+b}{d}\right) p^m \\
&= -\sum_{m>0} T_mf(g,1,\tau) p^m \\
&= -\sum_{m>0} T_mf(1,g,-1/\tau) p^m
\end{aligned} \]

By Lemma \ref{lem:logp}, this sum agrees with the expansion of $\log p(f(1,g,\sigma) - f(1,g,\tau))$ at $(i\infty,0)$ as elements of the formal power series ring $\C((q^{1/N}))[[p]]$.  The function $p(f(1,g,\sigma) - f(1,g,\tau))$ has no zeroes or poles in a neighborhood of the line $\{ (\sigma, \tau) \in \HH^* \times \HH |\sigma = i\infty, \tau \in i\R_{>0} \}$, so $\log p(f(1,g,\sigma) - f(1,g,\tau))$ is single-valued and holomorphic in a neighborhood of the same line.  To show that the two cusp expansions at $(i\infty,i\infty)$ and $(i\infty,0)$ are analytic continuations of each other, it suffices to show that they converge on neighborhoods whose union contains this line.  We note that by a standard theorem of several complex variables (e.g., \cite[Theorem B.2]{G90}), the domain of convergence is always a complete Reinhardt domain with log-convex base.

Because the function $\tau \mapsto f(1,g,\tau)$ has a simple pole at infinity and is holomorphic on $\HH$, we can use the following injectivity property: for any $y \in \R_{>0}$, there exists $\sigma$ with sufficiently large imaginary part that any $\tau \in \HH$ satisfying $\Im(\tau)>y$ and  $f(1,g,\tau) = f(1,g,\sigma)$ also satisfies $\tau - \sigma \in \Z$.  Equivalently, given any $y$, there exists $\alpha \in \R_{>0}$ such that the power series expansion of $\log p(f(1,g,\sigma) - f(1,g,\tau))$ converges in the region $\{ (p,q) \mid |p| < e(i\alpha), |q| < e(iy) \}$, i.e., on a polydisc.  Since $y$ is arbitrary, we have convergence on the complement of an arbitrarily small segment of the line $\{ (\sigma, \tau) \in \HH^* \times \HH |\sigma = i\infty, \tau \in i\R_{>0} \}$.  This reduces the problem to showing that the power series expansion of $\log p(f(1,g,\sigma) - f(g,1,\tau))$ converges in a nontrivial neighborhood of the cusp $(i\infty,i\infty)$.

\noindent\textbf{Case 1:} Suppose $f(g,1,\tau)$ has a pole at infinity, i.e., $f(1,g,\tau)$ has a pole at zero.  Corollary 3.6 in \cite{C10} asserts that $f(1,g,\tau)$ has global symmetries, i.e.,, there exists some element $\binom{ab}{cd} \in PSL_2(\R)$ that takes zero to infinity, such that $f(1,g,\tau) = f(1,g,\frac{a\tau+b}{c\tau+d})$ for all $\tau \in \HH$.  Since zero is taken to infinity, we have $d=0$, and the fact that $f$ is a Hauptmodul implies the transformation must have the form $\tau \mapsto a - b/\tau$ for $b | N$ and $a \in \Q$.  Then using Lemma \ref{lem:logp}, 
\[ \begin{aligned}
-\sum_{m>0} T_mf(1,g,-1/\tau) p^m &= -\sum_{m>0} T_mf(1,g,a - \frac{b}{-1/\tau}) \\
&= -\sum_{m>0} T_mf(1,g,a+b\tau) \\
&= \log p(f(1,g,\sigma) - f(1,g,a+b\tau)) 
\end{aligned} \]
$a$ is real and $b$ is nonzero, so the last line is the sum of $\log (1 - e(-a)pq^{-b})$ and a power series that converges on a nontrivial neighborhood of $p=q=0$.  This implies $\log p(f(1,g,\sigma) - f(g,1,\tau))$ converges near $(i\infty,i\infty)$.

\noindent\textbf{Case 2:} Suppose $f(g,1,\tau)$ is regular at infinity, i.e., $f(1,g,\tau)$ is regular at zero.  By continuity, there exists a neighborhood $U \subset \HH \times \HH$, whose closure contains the cusp $(i\infty,0)$, satisfying $f(1,g,\sigma) \neq f(1,g,\tau)$ for all $(\sigma,\tau) \in U$.  The function $\log p(f(1,g,\sigma) - f(1,g,\tau))$ is therefore single-valued and holomorphic on $U$ and has a unique power series expansion that converges absolutely uniformly on some open subset of $U$ containing $(i\infty,0)$.
\end{proof}


\subsection{Replicable functions and products} \label{sec:replicable}

This section is not essential for the main results, but it describes a direct relationship between replicability of series and the form of infinite products.

\begin{defn} (\cite{N84}) For any formal Laurent series $f(q) = \sum_{n \in \Z} c(n) q^n \in \C((q))$ of the form $q^{-1} + O(q)$, one has the bivarial transform, which yields numbers $H_{m,n}$ by the formal expansion $\log \frac{f(p) - f(q)}{p^{-1}-q^{-1}} = - \sum_{m,n = 1}^\infty H_{m,n} p^m q^n$.  The series $f$ is called replicable if all of its coefficients are integers and the value of $H_{m,n}$ depends only on $mn$ and $(m,n)$.  A holomorphic function on the upper half plane is replicable if it admits a $q$-expansion that is a replicable series.
\end{defn}

\begin{lem} \label{lem:prod-gcd}
A Laurent series $f$ with integer coefficients is replicable if and only if one has the following formal infinite product expansion
\[ f(p) - f(q) = p^{-1} \prod_{m>0, n \geq -1} (1-p^m q^n)^{c(m,n)} \]
such that the exponents $c(m,n)$ depend only on $mn$ and $(m,n)$.
\end{lem}
\begin{proof}
The presence of the $p^{-1}$ in the front of the product restricts the form of $f$ to have a simple pole of residue 1 at $q=0$.  Taking the formal logarithm of both sides of the product formula yields:
\[ \begin{aligned}
\log(f(p)-f(q)) &= \log(p^{-1} - q^{-1}) - \sum_{i>0} \sum_{m,n = 1}^\infty \frac{c(m,n)}{i} p^{mi} q^{ni} \\
&= \log(p^{-1} - q^{-1}) - \sum_{m,n=1}^\infty p^m q^n \sum_{t|(m,n)} c(m/t,n/t)/t.
\end{aligned} \]

We see that $H_{m,n} = \sum_{t|(m,n)} c(m/t,n/t)/t$, and note that if $c(m,n)$ only depends on $mn$ and $(m,n)$, then the same is true for $H_{m,n}$.  We can use M\"obius inversion to reverse this process:

\[ \begin{aligned}
\sum_{m,n=1}^\infty H_{m,n} p^m q^n &= \sum_{i=1}^\infty \frac{1}{i} \sum_{m,n = 1}^\infty c(m,n) p^{mi} q^{ni} \\
\sum_{t=1}^\infty \frac{\mu(t)}{t} \sum_{m,n=1}^\infty H_{m,n} p^{mt} q^{nt} &= \sum_{i,t=1}^\infty \frac{\mu(t)}{it} \sum_{m,n = 1}^\infty c(m,n) p^{mit} q^{nit} \\
&= \sum_{s=1}^\infty \sum_{it=s} \frac{\mu(t)}{s} \sum_{m,n = 1}^\infty c(m,n) p^{ms} q^{ns} \\
\sum_{m,n=1}^\infty p^m q^n \sum_{t|(m,n)} \frac{\mu(t)}{t} H_{m/t,n/t} &= \sum_{m,n = 1}^\infty p^m q^n \sum_{s|(m,n)} \frac{c(m/s,n/s)}{s} \sum_{t|s} \mu(t) \\
\sum_{t|(m,n)} \frac{\mu(t)}{t} H_{m/t,n/t} &= \sum_{s|(m,n)} \frac{c(m/s,n/s)}{s} \delta_{s,1} \\
&= c(m,n)
\end{aligned} \]

By the same reasoning as before, if each $H_{m,n}$ depends only on $mn$ and $(m,n)$, then the same holds for $c(m,n)$.
\end{proof}

The notion of replicable function was originally introduced in \cite{CN79}, where it was defined in terms of explicit relations between coefficients of the $q$-expansion.  This interpretation was formalized in \cite{N84} as the existence of ``replicate functions'', and a proof of equivalence was sketched.  Given a formal Laurent series of the form $f(q) = q^{-1} + O(q) = \sum_{n \in \Z} c(n) q^n \in \Z((q))$, and an integer $n>0$ there is a unique normalized Faber polynomial $\Phi_n(x)$ (depending on $f$), defined by the property that $\Phi_n(f(q)) = q^{-n} + O(q)$.  Since the polynomial in $f(q)$ that is $n$ times the coefficient of $p^n$ in the formal series $- \log p(f(p) - f(q))$ also has this form, we have $\Phi_n(f(q)) = q^{-n} + n\sum_m H_{m,n}q^m$.  It is straightforward to show that $f(q)$ is replicable if and only if for any $t \in \Z_{>0}$ there exists a series $f^{(t)}(q) = q^{-1} + O(q)$ such that $\Phi_n(f(q)) = \sum_{ad=n, 0 \leq b<d} f^{(a)}(e(b/d)q^{a/d})$.  The series $f^{(t)} = \sum_{n>0} a_n^{(t)} q^n$ is called the $t$-th replicate of $f$, and by suitable use of M\"obius inversion, one can show that it is unique if it exists, with coefficients satisfying the relations:
\[ H_{m,n} = \sum_{t|(m,n)} \frac{1}{t} a_{mn/t^2}^{(t)} \qquad \text{and} \qquad a_k^{(t)} = t \sum_{s|t} \mu(s) H_{kts,t/s}. \]

\begin{defn} A replicable series $f$ is has order $N$ if $f^{(n+N)} = f^{(n)}$, for all $n > 0$.
\end{defn}

In Proposition 5.3 of \cite{C10}, we gave an interpretation of replicability of a function as the condition that all equivariant Hecke operators $T_n$ applied to the function are equal to certain polynomials applied to the function.  In particular, Lemma 5.2 identifies the series $\Phi_n(f(q))$ with $nT_n f(q)$, and Corollary 5.4 asserts that finite order replicable functions are either Hauptmoduln invariant under some $\Gamma_1(N)$ or have the form $q^{-1}$, $q^{-1}-q$, or $q^{-1} + q$.

\begin{defn}  A replicable series $f$ is completely replicable if for all $n >0$, the series $f^{(n)}$ is also replicable.
\end{defn}

If $f$ is replicable, then the condition that $f$ is completely replicable is equivalent to the (a priori stronger) condition that $(f^{(s)})^{(t)}$ exists and is equal to $f^{(st)}$ for all $s,t>0$ \cite{K94}.  Theorem 1.3 of \cite{CG97} asserts that completely replicable functions of finite order are either Hauptmoduln, or have finite expansions of the form $q^{-1}$, $q^{-1} + q$, or $q^{-1} - q$.  By Corollary 5.6 in \cite{C10}, there is a natural bijection between principally normalized Hecke-monic functions on $\M^{\Z/N\Z}_{Ell}$ and completely replicable modular functions of order dividing $N$.

\begin{lem}
If $f$ is a completely replicable series of order $N$, then the exponents $c(m,n)$ in the formal product expansion of $f(p)-f(q)$ depend only on $mn$ and $(m,n,N)$.
\end{lem}
\begin{proof}
Suppose $f$ is completely replicable of order $N$.  From the previous discussion, we know that $H(m,n) =  \sum_{t|(m,n)} c(m/t,n/t)/t = \sum_{t|(m,n)} \frac{1}{t} a_{mn/t^2}^{(t)}$.  M\"obius inversion then yields $c(m,n) = \sum_{st|(m,n)} \frac{\mu(s)}{st} a_{mn/s^2t^2}^{(t)}$.  Following the arguments after equation (10.4) of \cite{B92}, we note that if $b$ is any function on positive integers such that $b(n)$ depends only on $(n,N)$ for some $N$, then $\sum_{st=k} \mu(s)b(t)$ is zero unless $k|N$.  Since the coefficients $a_n^{(t)}$ depend only on $n$ and $(t,N)$, the sum $\sum_{st = k} \mu(s)a_{mn/s^2t^2}^{(t)}$ vanishes unless $k|N$.  In particular, we obtain the simplified equation $c(m,n) = \sum_{st|(m,n,N)} \frac{\mu(s)}{st} a_{mn/s^2t^2}^{(t)}$, which depends only on $mn$ and $(m,n,N)$.
\end{proof}

We suspect that a converse holds as well, i.e., that the conclusion of this lemma gives a characterization of finite order completely replicable functions.

\subsection{Singular theta-lifts}

We rephrase the product formula results of Section \ref{sec:hecke-monic} in terms of Borcherds-Harvey-Moore singular theta-lifts.  This section is not strictly necessary for the Lie algebra constructions in this paper, but the formalism allows us to consider product expansions at cusps other than $(i\infty,i\infty)$ and $(i\infty,0)$.  These are often Weyl denominators for infinite dimensional Lie superalgebras, but it is not at all clear what physical interpretation they should have.

We introduce some geometry, borrowing heavily from \cite{B98}.  Let $O_{2,2}$ be the automorphism group of the quadratic form $Q(a,b,c,d) \mapsto ab + cd$.  We employ the transpose-inverse outer automorphism of $SL_2$ to construct two commuting embeddings $SL_2 \to O_{2,2}$ via:
\[ \left( \begin{array}{cc} a & b \\ c & d \end{array} \right) \mapsto \left( \begin{array}{cccc} a & 0 & b & 0 \\ 0 & d & 0 & -c \\ c & 0 & d & 0 \\ 0 & -b & 0 & a \end{array} \right) \text{ or } \left( \begin{array}{cccc} a & 0 & 0 & b \\ 0 & d & -c & 0 \\ 0 & -b & a & 0 \\ c & 0 & 0 & d \end{array} \right) \]
This induces an injection $SL_2(\Z) \times SL_2(\Z) \to \Aut(I\! I_{1,1} \times I\! I_{1,1})$ as an index four subgroup.  If we view $M = I\! I_{1,1}(N) \times I\! I_{1,1}$ as a sublattice of $I\! I_{1,1} \times I\! I_{1,1}$ in these coordinates as $\Z \times N\Z \times \Z \times \Z$, the above map pulls back to $\Gamma_0(N) \times \Gamma_0(N) \to \Aut(M)$.  The Grassmannian $G(M) \cong O_{2,2}(\R)/O_2(\R) \times O_2(\R)$ is the space of two dimensional positive definite subspaces of $M \otimes \R$.  If we choose a point in $G(M)$ and an ordered basis $X_M,Y_M$ of the corresponding positive definite plane, then $X_M + iY_M \in M \otimes \C$.  The vector $X_M + iY_M$ has norm zero if and only if $X_M$ and $Y_M$ are orthogonal and of equal norm.  There is an action of $\C^\times$ on the space of ordered orthogonal bases whose vectors have equal positive norm, by simultaneous rotations and dilations.  The manifold $G(M)$, which is isomorphic as a hermitian symmetric domain to $\HH \times \HH$, can be viewed as an analytic open subset of the quadric surface of all norm zero vectors in $\mathbb{P}(M \otimes \C)$.  In particular, the subspace of $\mathbb{P}(M \otimes \C)$ for which $X_M$ (equivalently, $Y_M$) has positive norm has two connected components, and we will choose one of them to identify with $G(M)$, placing an orientation on the plane spanned by $X_M$ and $Y_M$.  Each copy of $\Gamma_0(N)$ above acts on $M$ and $M \otimes \C$ by norm-preserving transformations, and the product $\Gamma_0(N) \times \Gamma_0(N)$ in $\Aut(M)$ acts on the norm zero vectors in the projectivization, preserving $G(M)$.  The canonical $\C^\times$-torsor over $\mathbb{P}(M \otimes \C)$ restricts to a $\C^\times$-torsor $P$ over $G(M)$, and since $G(M)$ is contractible, we can trivialize the associated line bundle.  Trivializations are not equivariant under the action of $\Gamma_0(N) \times \Gamma_0(N)$, and we say that a function on $G(M)$ is an automorphic form of weight $k$ if under an inverse trivialization of the the $k$th tensor power of this line bundle, the function becomes an invariant section.  We are concerned with automorphic functions, i.e, weight zero forms.

We choose once and for all an identification $G(M) \cong \HH \times \HH$ by sending the $\C^\times$-span of the norm zero vector $(i,i,1,1) \in M \otimes \C$ to $(i,i) \in \HH \times \HH$.  We can specify the rest of the points by M\"obius transformations on each copy of $\HH$, essentially using Iwasawa decomposition: $\begin{pmatrix} \sqrt{y} & x/\sqrt{y} \\ 0 & 1/\sqrt{y} \end{pmatrix}$ sends $i$ to $x+iy$, so by using the above embeddings $SL_2 \to O_{2,2}$ and rescaling appropriately, $(z_1,z_2) = (x_1+iy_1,x_2+iy_2)$ corresponds to the $\C^\times$-span of:
\[ (-y_1 y_2 + x_1 x_2 + ix_1y_2 + ix_2 y_1, -1, x_1+ iy_1, x_2 + i y_2) = (z_1z_2, -1, z_1,z_2) \]

\begin{lem} 
Given a vector $\lambda = (a,b,c,d) \in M \otimes \R$ of negative norm,
\[ \lambda^\perp = \{ (z_1,z_2) \in \HH \times \HH | z_1 = \frac{-cz_2 + a}{bz_2 + d}  \} \]
\end{lem}
\begin{proof}
Orthogonality with $(a,b,c,d)$ is equivalent to the condition $-a+bz_1z_2 + cz_2+dz_1 = 0$.  Solving for $z_1$ yields the answer.
\end{proof}
\noindent\textbf{Remark:} The negative norm condition on $\lambda$ is equivalent to the corresponding matrix $\begin{pmatrix} -c & a \\ b & d \end{pmatrix}$ having positive determinant, so the condition is necessary for the set $\lambda^\perp$ to be nonempty.

Cusps are represented by norm zero vectors in $M$, and they can be produced by taking limits of M\"obius transformations - in particular, the cusp at $(1,0,0,0)$ corresponds to $(i\infty, i\infty)$ under our identification.  The orbits of cusps in $\HH \times \HH$ under the action of  $\Gamma_0(N) \times \Gamma_0(N)$ naturally map to $\Aut M$ orbits of primitive norm zero vectors in $M$.  Cusp generation can be done explicitly as follows: Any $a/c \in \mathbb{P}^1(\Q)$ is the image of $i\infty$ under some $\tau \mapsto \frac{a\tau + b}{c\tau + d}$ in $SL_2(\Z)$, so the cusp at $(a_1/c_1, a_2/c_2)$ corresponds to the vector:
\[ \left( \begin{array}{cccc} a_1 & 0 & b_1 & 0 \\ 0 & d_1 & 0 & -c_1 \\ c_1 & 0 & d_1 & 0 \\ 0 & -b_1 & 0 & a_1 \end{array} \right) \left( \begin{array}{cccc} a_2 & 0 & 0 & b_2 \\ 0 & d_2 & -c_2 & 0 \\ 0 & -b_2 & a_2 & 0 \\ c_2 & 0 & 0 & d_2 \end{array} \right) \left( \begin{array}{c} 1 \\ 0 \\ 0 \\ 0 \end{array} \right) = 
\left( \begin{array}{c} a_1 a_2 \\ -c_1 c_2 \\ c_1 a_2 \\ a_1 c_2 \end{array} \right) \]
with a sign ambiguity.  To go back, we write a primitive norm zero vector as $(a,b,c,d)$.  Using the above vector representation, we find that the corresponding cusp is $(a/c = -d/b, -c/b = a/d)$, where we choose fractions that don't have the form $0/0$.  The following four cusps will be most important to us:
\[ \begin{aligned}
\pm (1,0,0,0) &\leftrightarrow (i\infty,i\infty) \\
\pm (0,1,0,0) &\leftrightarrow (0,0) \\
\pm (0,0,1,0) &\leftrightarrow (0,i\infty) \\
\pm (0,0,0,1) &\leftrightarrow (i\infty,0)
\end{aligned} \]

\subsubsection{Lorentzian case}

Borcherds's strategy for evaluating theta-lifts for orthogonal groups of indefinite signature is to split off hyperbolic subspaces of the real span of the lattice until one gets a definite space, for which the Grassmannian is a point, then build up lifts for increasingly larger lattices in terms of the lifts for the smaller ones.

Given a primitive norm zero vector $z \in M$, we form a singular lattice $z^\perp \subset M$, and a hyperbolic lattice $K = z^\perp/\Z z$.  We choose a norm zero $z' \in M \otimes \R$ such that $(z,z')=1$, and this lets us identify $K$ with a sublattice of $M \otimes \R$ that lies in the orthogonal complement of $z$ and $z'$.  For example, when $z = (0,0,1,0)$ , we can choose $z' = (0,0,0,1)$, and this identifies $K$ with $I\! I_{1,1}(N) \times \{ 0 \} \subset M$.  If instead $z = (1,0,0,0)$, we can choose $z' = (0,1,0,0) \in M^\vee$, identifying $K$ with $\{ 0 \} \times I\! I_{1,1} \subset M$.

Once we have chosen $z$ and $K$, we can make a second parametrization of $G(M)$ following \cite{B98} section 13.  Given an oriented basis $(X_M,Y_M) \in G(M)$ satisfying $(Y_M,z) = 0$, $(X_M,z)>0$, there exists a positive norm vector $Y \in K \otimes \R$ such that $Y_M$ is a sum of $Y$ and a multiple of $z$.  $Y$ is determined only by the choice of $z$ and the choice of orientation of the basis.  The set of positive norm vectors in $K \otimes \R$ has two connected components, and we choose a cone $C$ to be the component corresponding to our choice of orientation.  This identifies $G(M)$ with $K \otimes \R + iC \subset K \otimes \C$, and this identification is a special case of the general fact that the Grassmannian for $O_{2,n-2}$ is isometric to the Cartesian product of the Lorentz space $\R^{1,n-3}$ with its positive light cone.  There is a map $K \otimes \R + iC \to P$ given by letting $Z = X + iY \in K \otimes \R + iC$ and setting $Z_M := (Z,-1,Q(Z)) \in K \otimes \C \oplus \C z' \oplus \C z \cong M \otimes \C$ to be the unique norm zero point having inner product $1$ with $z$ and projection $Z$.

Inside $K$, we choose a primitive norm zero vector $w$, and another norm zero vector $w' \in K^\vee$ satisfying $(w,w')=1$.  Following \cite{B98} section 5, we can construct functions $F_M$, $F_K$, and $F_0$ by combining individual components of $F$:

\[ \begin{aligned}
F_M &= \sum_{\delta \in M^\vee \! /M} F_\delta e_\delta = F \\
F_K &= \sum_{\lambda \in K^\vee/K} e_\lambda \sum_{\delta \in M^\vee \! /M, \delta|z^\perp = \lambda} F_\delta \\
F_0 &= \sum_{\lambda \in K^\vee/K, \lambda|(K \cap w^\perp) = 0} \sum_{\delta \in M^\vee \! /M, \delta|(M \cap z^\perp) = \lambda} F_\delta \\
&= \sum_{\delta \in (M^\vee \cap (\Q w + \Q z))/M} F_\delta
\end{aligned} \]

The notation $\delta|z^\perp = \lambda$ means that we lift $\lambda$ to a set of linear functions $K \to \Z$, pull them back along the projection $z^\perp \to K$, lift $\delta$ to a set of linear functions $M \to \Z$, restrict them to $z^\perp$, and compare the elements of the two sets.

\begin{lem} \label{lem:F0form} 
Let $f$ be a principally normalized function on $\M_{Ell}^{\Z/N\Z}$, and let $F = F_M$ be the corresponding vector-valued function.  If $z = (1,0,0,0)$, then $F_K$ is a scalar function of the form $q^{-1} + O(q)$, and $F_0$ has the form $q^{-1} + O(q)$.  If $z = (0,0,1,0)$ or $(0,0,0,1)$, and if $w = (1,0,0,0)$, then $F_K$ has $N^2$ components, the $N$ components in the span of $w$ have the form $q^{-1} + O(q)$ (if it is the zero component) or $O(q)$ (otherwise), and $F_0$ has the form $q^{-1} + O(q)$.
\end{lem}
\begin{proof}
If $z = (1,0,0,0)$, $z^\perp = \{ (a,0,c,d) \} \subset M$ and $K^\vee/K$ is trivial, so the sum is over $\lambda = 0$ and $\delta \in \{ (a,0,c,d) + M \} \subset M^\vee/M$.  Using the identification with the coordinate switch, we have $F_K = \sum_{a \in \Z/N\Z} F_{0,a} = \hat{F}_{0,0}$.  Since $f$ is principally normalized, this has the form $q^{-1} + O(q)$.

If $z = (0,0,1,0)$, then $z^\perp = \{ (a,b,c,0) \} \subset M$, and $K^\vee/K \cong \Z/N\Z \times \Z/N\Z$.  If $z = (0,0,0,1)$, then $z^\perp = \{a,b,0,d \} \subset M$, and $K^\vee/K$ is the same.  In either case, each component of $F_K$ is defined by a sum over $\delta$ that has exactly one summand.  By the identification of $M^\vee/M$ with $\Z/N\Z \times \Z/N\Z$, the components in the span of $w$ have the form $F_{0,a}$ for $a \in \Z/N\Z$.  Since $f$ is principally normalized, we have $\hat{F}_{0,k} = q^{-1} + O(q)$ for all $k$, and $F_{0,a} = \frac{1}{N} \sum_{k \in \Z/N\Z} e(-ak/N) \hat{F}_{0,k} = \delta_{0,a} q^{-1} + O(q)$.  Unwinding the sum, we find that $F_0 = \sum_{a \in \Z/N\Z} F_{0,a} = \hat{F}_{0,0} = q^{-1} + O(q)$.
\end{proof}

\begin{lem} 
If $F_0 = q^{-1} + O(q)$, then the singular theta-lift for the zero lattice is $\Phi_0 = -8\pi$.
\end{lem}
\begin{proof}
By \cite{B98} Corollary 9.3, $\Phi_0$ is $\frac{\pi}{3}$ times the constant term of $F_0 E_2$.  We multiply $q^{-1} + O(q)$ by $1 - 24q + O(q^2)$ to get a constant term of -24.
\end{proof}

To describe the singular theta-lift $\Phi_K$ for the Lorentzian lattice, we need to introduce some notation.  $G(K)$ is the space of positive-definite lines in $\R^{1,1}$, and can be identified with one connected component of the hyperbola of the norm one vectors $v \in K \otimes \R$.  We write $w_{v^+}$ to denote the orthogonal projection of the norm zero vector $w$ to the line $v^+$ spanned by $v$, $c_\lambda(0)$ describes the constant terms in $F_K$, and $\mathbf{B}_2(x)$ is the periodic function that equals $x^2-x+1/6$ when $0 \leq x \leq 1$.  Then by \cite{B98} Theorem 10.2, we have
\[ \sqrt{2} |w_{v^+} |\Phi_K (v,F_K) = \Phi_0 + 4 \pi w_{v^+}^2 \sum_{\lambda \in K^\vee/K, \lambda|(K \cap w^\perp) = 0} c_\lambda(0) \mathbf{B}_2((\lambda, w')) \]
By \cite{B98} Theorem 6.3, $\Phi_K$ is a piecewise linear function on the positive cone of $G(K)$ in $K \otimes \R$.

\begin{lem} \label{lem:phik} 
If $F_0 = q^{-1} + O(q)$ and none of the components of $F_K$ has a nonzero constant term, then $\Phi_K(v,F_K) = -8\sqrt{2}\pi |w'_{v^+}|$.
\end{lem}
\begin{proof}
Since the constant terms $c_\lambda(0)$ vanish, the sum in the above formula is zero.  Hence, $\sqrt{2} |w_{v^+} |\Phi_K (v,F_K) = \Phi_0$, and this is $-8\pi$.  We can write $v = mw' + nw$, so $v^2 = 2mn = 1$, $w_{v^+} = mv$, $|w_{v^+}| = m$, and $|w'_{v^+}| = n = 1/2m$.  Then $\Phi(mw'+nw,F_K) = -8\sqrt{2} \pi n$.
\end{proof}

Following \cite{B02}, we define for each $\lambda \in K^\vee, n<0$, a distinguished set of points $H(\lambda,n) = \{ \lambda^\perp | c_{\lambda}(n) \neq 0, Q(\lambda) = n\}$.  The complement of $\bigcup_{\lambda, n} H(\lambda, n)$ in $G(K)$ is a finite disjoint union of segments.  Their positive cones in $K\otimes \R$ are called the Weyl chambers in K, and $\Phi_K$ is real analytic (in fact a piecewise polynomial) when restricted to a Weyl chamber.  Given a Weyl chamber $W$, we write $(\lambda,W)>0$ if the inner product of $\lambda$ with any vector in the interior is positive.  The set $\{ \lambda \in K^\vee | c_\lambda(n) \neq 0$ for some $n \}$ is contained in the disjoint union of $\{ \lambda \in K^\vee | (\lambda,W)>0 \}$ with $\{ \lambda \in K^\vee | (-\lambda, W) > 0 \}$.

For each primitive norm zero vector $z \in M$ (equivalently, each cusp) and each Weyl chamber $W$, a Weyl vector $\rho(K,W,F_K)$ can be defined by $8\sqrt{2}\pi(\rho(K,W,F_K),v) = |v|\Phi_K(v/|v|,F_K)$.  By \cite{B98} theorem 10.4, we have $\rho(K,W,F_K) = \rho_w w + \rho_{w'} w'$, where
\[ \begin{aligned}
\rho_{w'} &= \text{constant term of } F_0(\tau) E_2(\tau)/24 \\
\rho_w &= \frac{1}{4} \sum_{\lambda \in (K^\vee \cap \Q w)/K} c_\lambda(0) \mathbf{B}_2((\lambda, w'))
\end{aligned} \]

\begin{lem} \label{lem:weylv} 
If $F_0 = q^{-1} + O(q)$ and none of the components of $F_K$ in the span of $w$ has a nonzero constant term, then $\rho(K,W,F_K) = -w'$.
\end{lem}
\begin{proof}
Since $E_2(\tau) = 1 - 24q + \dots$, we have $\rho_{w'} = -1$.  Since all of the $c_\lambda(0)$ vanish in the sum, $\rho_w = 0$.
\end{proof}

\subsubsection{Signature 2,2}

We now consider the singular theta-lift $\Phi_M$.  This is a real-analytic function on $G(M) \cong \HH \times \HH$ with logarithmic singularities along certain divisors of the form $\lambda^\perp$, as described in theorem 6.2 of \cite{B98}.  We are more interested in holomorphic functions $\Psi$ satisfying $\Phi_M = -4\log|\Psi|$, because $\Psi$ may be highly invariant and has zeroes and poles along the same divisors.

\begin{thm} \label{thm:prods} 
Let $f$ be a principally normalized Hecke-monic function on $\M_{Ell}^{\Z/N\Z}$, and let $F$ be the corresponding vector-valued function of type $\rho_M$.  Suppose at least one of the following two conditions are satisfied:
\begin{enumerate}
\item $c_{0,k}(n) \in \R$ for all $k \in \Z/N\Z$ and all $n \in \Q$.
\item $c_{i,k}(n) \in \R$ for all $n < 0$, and all $i,k \in \Z/N\Z$.
\end{enumerate}
Then the function $f(1,g,\sigma) - f(1,g,\tau)$ satisfies the following properties:
\begin{enumerate}
\item It is invariant under the the standard action of $\Gamma_0(N) \times \Gamma_0(N)$ on $\HH \times \HH$.
\item The only zeroes and poles lie on rational quadratic divisors $\lambda^\perp$ for $\lambda \in M$, $Q(\lambda)<0$ and are zeroes of order
\[ \sum_{0<x \in \R, x\lambda \in M^\vee} c_{x\lambda} (Q(\lambda)x^2). \]
\item For each primitive norm zero vector $z \in M$, let $m_z$ denote the smallest positive value of the inner product of $z$ with a vector in $M$.  For each Weyl chamber $W$ of $K$, the restriction of $f(1,g,\sigma) - f(1,g,\tau)$ has an infinite product expansion converging when $Z = X + iY \in K \otimes \R + iC \subset K \otimes \C$ is in a neighborhood of the cusp of $z$ and $Y \in W$.  The product is some constant of absolute value
\[ \underset{k \neq 0}{\prod_{k \in \Z/m_z\Z}} (1-e(k/m_z))^{c_{k z/m_z}(0)/2} \]
times
\[ e((Z,\rho(K,W,F_K))) \prod_{\lambda \in K^\vee, (\lambda,W)>0} \prod_{\delta \in M^\vee \! /M, \delta|z^\perp = \lambda} (1-e((\lambda,Z) + (\delta,z')))^{c_\delta(Q(\lambda))} \]
\end{enumerate}
Furthermore, if we let $e((w',Z)) = p$ and $e((w,Z)) = q$, then the expansions of $f(1,g,\sigma) - f(1,g,\tau)$ at $z=(i\infty,i\infty)$ and $(i\infty,0)$ are given by the product formulas from Lemma \ref{lem:prod1} and Proposition \ref{prop:prod2}:
\[ \begin{aligned}
f(1,g,\sigma) - f(1,g,\tau) &= p^{-1} \prod_{m > 0, n \in \Z} \prod_{j \in \Z/N\Z} (1 - e(j/N)p^m q^n)^{c_{0,j}(mn)} \\
f(1,g,\sigma) - f(g,1,\tau) &= p^{-1} \prod_{m>0, n \in \frac{1}{N}\Z} (1 - p^m q^n)^{c_{m,Nn}(mn)} .
\end{aligned} \]

\end{thm}
\begin{proof}
Since $f$ is principally normalized, the constant terms of $F_{0,k}$ vanish for all $k$, so in particular they are rational.  If the first condition holds, Lemma \ref{lem:mcgraw1} applies, and if the second condition holds, Lemma \ref{lem:mcgraw2} applies, so in either case, all coefficients of $F$ are real.

Since the $q$-expansion coefficients $c_\lambda(n)$ of $F$ are real, and $c_0(0) = 0$, then the arguments in the proof of \cite{B98} Theorem 13.3 show that there is a multivalued function $\Psi_M(Z_M,F)$ that is holomorphic for $Z_M$ in an open dense subset of $P := \{ Z \in M \otimes \C | Q(Z) = 0, (Z, \overline{Z}) > 0 \}$ with the following properties:
\begin{enumerate}
\item $\Psi_M(Z_M,F)$ is invariant under the scaling action of $\C^\times$ on $P$, i.e., it descends to a multivalued function $\Psi$ on an open dense subset of $G(M)$.  $\Psi$ is invariant up to phase under $\Aut(M,F) \supseteq \Gamma_0(N) \times \Gamma_0(N)$.
\item The only zeroes and poles of $\Psi$ lie on rational quadratic divisors $\lambda^\perp$ for $\lambda \in M$, $Q(\lambda)<0$ and are zeroes of order
\[ \sum_{0<x \in \R, x\lambda \in M^\vee} c_{x\lambda} (Q(\lambda)x^2). \]
If for some $\lambda$ this sum is not an integer, there is a branch on $\lambda^\perp$ with monodromy given by exponentiating $2 \pi i$ times this sum.  $\Psi$ has single-valued holomorphic branches on any simply-connected neighborhood that has trivial intersection with the locus where this sum is not an element of $\N$.
\item For each primitive norm zero vector $z \in M$, let $m_z$ denote the smallest positive value of the inner product of $z$ with a vector in $M$.  For each Weyl chamber $W$ of $K$, the restriction $\Psi_z(Z,F)$ has an infinite product expansion converging when $Z = X + iY \in K \otimes \R + iC \subset K \otimes \C$ is in a neighborhood of the cusp of $z$ and $Y \in W$.  The product is some constant of absolute value
\[ \underset{k \neq 0}{\prod_{k \in \Z/m_z\Z}} (1-e(k/m_z))^{c_{k z/m_z}(0)/2} \]
times
\[ e((Z,\rho(K,W,F_K))) \prod_{\lambda \in K^\vee, (\lambda,W)>0} \prod_{\delta \in M^\vee \! /M, \delta|z^\perp = \lambda} (1-e((\lambda,Z) + (\delta,z')))^{c_\delta(Q(\lambda))} \]
\end{enumerate}
The branch locus of $\Psi$ is nowhere dense, since it is contained in $\{ Z \in \lambda^\perp | \lambda \in M^\vee, 0 > Q(\lambda) > k \}$ for any negative $k$ that bounds the order of the poles of $F$.

The reader may notice a remarkable similarity between these properties of $\Psi$ given by Borcherds's theorem and the properties we claim for $f(1,g,\sigma) - f(1,g,\tau)$ in the statement of our theorem.  Indeed, our goal for the next two paragraphs is to show that $\Psi$ and $f(1,g,\sigma) - f(1,g,\tau)$ are equal up to multiplication by an invertible constant.

We consider the product expansion at the cusp $z = (i\infty,i\infty) = (1,0,0,0)$.  By Lemma \ref{lem:F0form} we have $F_K = q^{-1} + O(q)$ and $F_0 = q^{-1} + O(q)$.  We fix $w = (0,0,1,0)$ and $w' = (0,0,0,1)$, and Lemma \ref{lem:weylv} implies $\rho(K,W,F_K) = -w'$, so $e((Z,\rho(K,W,F_K))) = p^{-1}$.  The Weyl chambers are the two components of the complement of the hyperplane $(w-w')^\perp$, and we choose $W$ to be the Weyl chamber whose boundary contains the null vector $w$.  The set of $\lambda \in K^\vee$ satisfying $(\lambda,W)>0$ are then those of the form $mw' + nw$, with $m>0$.  Here, $m_z=N$, and all $F_{0,j}$ have no constant terms, so $c_{jz/N}(0) = c_{0,j}(0) = 0$, and the constant multiple has norm 1.  Choosing $z' = (0,1,0,0)$, we have:

\[ \begin{aligned}
& p^{-1} \prod_{m > 0, n \in \Z} \prod_{j \in \Z/N\Z} (1 - e(j/N)p^m q^n)^{c_{0,j}(mn)} \\
&= e((Z,-w')) \prod_{m>0,n\in\Z} \prod_{j \in \Z/N\Z} (1- e(\frac{j}{N}z,z') e((mw',Z)) e((nw,Z)))^{c_{jz}(mn)}  \\
&= e((Z, \rho(F_K,W,K))) \underset{(\lambda,W)>0}{\prod_{\lambda \in K^\vee}} \underset{\delta|M \cap z^\perp = \lambda}{\prod_{\delta \in M^\vee \! /M}} (1- e((\lambda,Z) + (\delta,z')))^{c_{\delta}(Q(\lambda))}
\end{aligned}\]

The function $\Psi$ has a unit magnitude constant ambiguity in its definition, so we may choose this constant to be one in a neighborhood of the cusp $z$.  Since the above product is equal to the expansion of $f(1,g,\sigma) - f(1,g,\tau)$ as shown in Lemma \ref{lem:prod1}, the two functions have identical analytic continuations, and $\Psi$ is therefore single-valued.  In particular, we find that the singular theta-lift is given by $\Phi_M = -4\log| f(1,g,\sigma) - f(1,g,\tau)|$.

We now consider the cusp $z = (i\infty,0) = (0,0,0,1)$.  If we choose $w = (1,0,0,0)$ and $w' = (0,1,0,0)$, Lemma \ref{lem:F0form} and Lemma \ref{lem:weylv} imply the Weyl vector is $-w'$, and $e((Z, \rho(K,W,F_K))) = p^{-1}$.  In this case, the Weyl chambers may be more complicated, since there may be more singular terms in $F_K$, but again we choose the chamber whose boundary contains $w$.  Here, $m_z=1$, so the constant multiple has unit magnitude.  Choosing $z' = (0,0,0,1)$ we have:
\[ \begin{aligned}
& p^{-1} \prod_{m > 0, n \in \frac{1}{N}\Z} (1 - p^m q^n)^{c_{m,Nn}(mn)} \\
&= e((Z,-w')) \prod_{m>0 , n \in \frac{1}{N}\Z} (1 - e(mw',Z) e(nw,Z))^{c_{mw'+nw}(mn)} \\
&= e((Z, \rho(F_K,W,K))) \underset{(\lambda,W)>0}{\prod_{\lambda \in K^\vee}} \underset{\delta|M \cap z^\perp = \lambda}{\prod_{\delta \in M^\vee \! /M}} (1- e((\lambda,Z) + (\delta,z')))^{c_{\delta}(Q(\lambda))}
\end{aligned} \]
This matches the product expansion given in Proposition \ref{prop:prod2}
\end{proof}

To fit the results about Hecke-monic product expansions from Section \ref{sec:hecke-monic} into the singular theta-lift framework, it now suffices to show that either the singularities of $F$ have real coefficients, or that $F$ has real coefficients along the span of $w$.  We do not know how to do this in general, but it can be proved for the cases that interest us most.

\begin{cor} 
Let $g$ be a generator of $\Z/N\Z$, and let $\pi: \Z/N\Z \to \MM$ be a homomorphism to the monster simple group, such that the McKay-Thompson series $T_{\pi(g)}(\tau)$ is invariant under $\Gamma_0(N)$.  Define functions $f_{(m)}$ by $f_{(m)}(\tau) = T_{\pi(g^m)}(\tau)$, and let $F$ be the corresponding vector-valued function of type $\rho_M$.  Then the first condition in Theorem \ref{thm:prods} is satisfied.  In particular, we have product expansions at all cusps with exponents given by coefficients of $F$, and the theta-lift $\Phi_M$ is equal to $-4\log| f(1,g,\sigma) - f(1,g,\tau)|$.
\end{cor}
\begin{proof}
The McKay-Thompson series satisfy $T_{\pi(g^i)}(\tau) = \mathrm{Tr}(g^i q^{L_0-1}|V^\natural)$, where $L_0-1$ is a grading operator on the monster vertex algebra $V^\natural$, which we can view as a graded vector space with a homogeneous monster action.  If $V_n$ is the subspace of $V^\natural$ on which $L_0-1$ acts by multiplication by $n$, then the $q^n$ coefficient of $T_{\pi(g^i)}(\tau)$ is $\mathrm{Tr}(g^i | V_n)$.  The Fourier transform takes traces of powers of $g$ to multiplicities of eigenvalues, so the coefficient $c_{0,k}(n)$ is the dimension of the subspace of $V_n$ on which $g$ acts by $e(k/N)$.  These dimensions are non-negative integers, so the first condition is satisfied.  The rest of the claims then follow from the theorem.
\end{proof}

\noindent\textbf{Note:} For the cases where $T_g(\tau)$ has squarefree level, this construction was suggested in Example 13.9 in \cite{B98}, and the product formulas in question were proved in \cite{S08}.  Scheithauer went a bit further in those cases, constructing multiple vector-valued forms, and writing down product expansions at all cusps.

\noindent\textbf{Remark:} There is a larger class of functions that we can fit into this framework, namely the completely replicable modular functions (see Section \ref{sec:replicable}).  Given such a function $f_{(1)}$ invariant under $\Gamma_0(N)$, we set $f_{(m)}$ to be its $m$th replicate, and by Corollary 5.6 in \cite{C10}, this defines a principally normalized Hecke-monic function on $\M_{Ell}^{\Z/N\Z}$.  It is more difficult to verify the first condition given in Theorem \ref{thm:prods}, since the functions $\hat{F}_{0,j}$ do not appear as characters of a previously constructed representation of a group.  Instead, we exhaustively computed the singular components of $\hat{F}$ using the description of all $n|h$-type invariance groups given in \cite{F93}, and verified that $c_{i,k}(n) \in \{ -1, 0, 1 \}$ for all $n < 0$.  Here $1$ (resp. $-1$) occurs if and only if a Fricke involution acts by $1$ (resp. $-1$) on $f_{(i)}$.  Since these coefficients are real, the second condition in Theorem \ref{thm:prods} holds.

\subsection{A monstrous example}

In this section, we describe an example of a product formula and an interpretation in vertex operator algebra language.  The construction of the product formula is purely function-theoretic, and uses only the Hauptmodul property of some functions found in \cite{CN79}.  However, our interpretation of these formulas in terms of algebraic structures rests on the construction in \cite{FLM88} of the monster vertex operator algebra $V^\natural$ and the corresponding interpretation of McKay-Thompson series as characters.

As we mentioned in the introduction, if we forget some of the structure of $V^\natural$, we have a vector space with commuting actions of a semisimple operator $L_0$, which endows it with a grading, and the monster simple group $\MM$, which acts homogeneously.  We fix an element $g$ of the monster simple group, and let $T_g(\tau) := \mathrm{Tr}(gq^{L_0-1}|V^\natural) \in \C((q))$ be the graded character.  We will call this the McKay-Thompson series for $g$, although this is not how the series were originally defined (since the conceptual interpretation as the character of a known object rests on the construction of $V^\natural$ in \cite{FLM88}).  For each $m \geq 1$, we define $f_{(m)}$ to be $\mathrm{Tr}(g^mq^{L_0-1}|V^\natural)$.  For any function $f$ on the upper half-plane, we define its eigengroup to be the subgroup of elements of $SL_2(\R)$ whose action on $\HH$ takes $f$ to a constant multiple of itself.  Conway and Norton \cite{CN79} introduced the notation $\Gamma_0(n|h)$ to denote the group of matrices
\[ \left\{ \begin{pmatrix} a & b/h \\ cn & d \end{pmatrix} \mid ad-bcn/h = 1, a,b,c,d \in \mathbb{Z} \right\}, \]
where $n=|g|$, and $h|(24,n)$.  The group $\Gamma_0(n|h)$ contains $\Gamma_0(nh)$, is conjugate to $\Gamma_0(n/h)$, and can be extended to $\Gamma_0(n|h)+s_1,s_2,\dots$ by adjoining some Atkin-Lehner involutions $W_{s_1}, W_{s_2},\dots$ for $s_i$ exactly dividing $nh$.

Conway and Norton conjectured that (under the original interpretation of McKay-Thompson series, without the existence of $V^\natural$) the set of these functions $\{ f_{(m)} \}_{m >0}$ satisfies the following property: $f_{(1)}$ has an eigengroup of the form $\Gamma_0(n|h)+s_1,s_2,\dots$ for some $h$, and $f_{(m)}$ has eigengroup $\Gamma_0(n'|h')+s'_1,s'_2,\dots$, where $n'=n/(n,m)$, $h'=h/(h,m)$, and $s'_i = s_i$ if and only if $s_i|\frac{n'}{h'}$ and $s'_i = 1$ otherwise.  Frenkel, Lepowsky, and Meurman reinterpreted this assertion in terms of characters of the vertex operator algebra $V^\natural$ they constructed, proving many cases \cite{FLM88}, and Borcherds proved it in full, using the same interpretation \cite{B92}.  A consequence of this fact is that (setting $N = nh$), if $f_{(1)}$ is fixed by $\Gamma_0(N)$, then $f_{(m)}$ is fixed by $\Gamma_0(N/(N,m))$.  Therefore, we can assemble these functions into a principally normalized Hecke-monic function on $\M_{Ell}^{\Z/N\Z}$, and take the Fourier transform to get a vector-valued function $F$.  Note that $N$ may be larger than $|g|$, since $h$ may be greater than one.  The eigengroups for all of the elements of the monster can be found in Table 2 of \cite{CN79}.

In this setting, we can call $\hat{F}$ the ``trace function'' and $F$ the ``multiplicity function'' since their components (conjecturally) describe the traces and multiplicities of eigenvalues for an action of $\Z/N\Z$ on the irreducible twisted modules of $V^\natural$.  The details of twisted modules are outside the scope of this paper, but we remark that the notion was introduced in a restricted form in \cite{FLM88}, and subsequently generalized in the work of several authors.  We will show in the next section that when we expand $\Psi$ at $(i\infty,0)$, the exponents in the product are the root multiplicities of a certain Lie algebra that we can define by generators and relations.  Assuming some additional conjectures, there is a variant of the Goddard-Thorn no-ghost theorem that implies the root spaces are naturally isomorphic to the eigenspaces of the corresponding twisted modules.  We hope to address this in future work.

We give a specific example to illuminate some of the objects in play.  We will focus on the case of an element in conjugacy class 3C, since that case is interesting for several reasons:
\begin{itemize}
\item The centralizer of 3C in $\MM$ is a direct product of $\Z/3\Z$ with Thompson's sporadic simple group $Th$, and since $Th$ has no nontrivial central extension \cite{CWCNP85}, the irreducible 3C-twisted module of $V^\natural$ has a natural action of $Th$.
\item The McKay-Thompson series for 3C is a three-fold dilation of the character of the basic representation of affine $E_8$, and the character of the 3C-twisted module is (conjecturally) a 3-fold compression of the same.  While $Th$ embeds in $E_8(\mathbb{F}_3)$, it does not embed in $E_8(\C)$ \cite{CG87}.  It is natural to ask whether the action of $Th$ on the 3C-twisted module extends to an action of $E_8(\mathbb{F}_3)$, and if so, precisely where the obstruction to a 3-adic lift resides.  This is analogous to an open problem in \cite{B98b} concerning a certain mod 3 vertex algebra ${}^gV$ attached to class 3C.
\item The dilation in the characters can be viewed as the 3-adic $U_3$ Hecke operator, arising from the Frobenius isogeny on mod 3 elliptic curves.  The ``Frobenius compression'' arises in the proof of Corollary 4.8 in \cite{B95c}, where it is used to endow the mod 3 vertex algebra ${}^gV$  with a conformal vector.  It would be interesting to have this apparent relationship between mod $p$ phenomena and twisted modules fleshed out more.
\item Most importantly for us, 3C is the conjugacy class of smallest order such that for an element $g$, the level of $T_g$ is strictly greater than the order of $g$.  The higher level creates a $\mu_3$ ambiguity in the $SL_2(\Z)$ compatibility formula in the generalized moonshine conjecture, since $T_g(-1/\tau)$ is not invariant under translation by 3.  This behavior is called a \textit{phase anomaly} in the mathematical physics literature.  When we consider centralizers of elements of $\MM$, this ambiguity manifests as the nontriviality of central extensions, i.e., in the analysis below, we are essentially pretending that $g$ generates a cyclic group of order 9.
\end{itemize}
Let $g$ be an element of order 3 in the conjugacy class 3C.  Its McKay-Thompson series is
\[ T_g(\tau) = q^{-1} + 248q^2 + 4124q^5 + 34752q^8 + \dots = j(3\tau)^{1/3}, \]
which has eigengroup
\[ \Gamma_0(3|3) = \begin{pmatrix} 1/3 & 0 \\ 0 & 1 \end{pmatrix} SL_2(\Z) \begin{pmatrix} 3 & 0 \\ 0 & 1 \end{pmatrix}, \]
i.e., this group acts by roots of unity.  In fact, $y_3 = \begin{pmatrix} 1 & 1/3 \\ 0 & 1 \end{pmatrix}$ acts by $e(-1/3)$, $x_3 = \begin{pmatrix} 1 & 0 \\ 3 & 1 \end{pmatrix}$ acts by $e(1/3)$, and the Fricke involution $\begin{pmatrix} 0 & -1 \\ 9 & 0 \end{pmatrix}$ fixes $T_g$.  The fixing group, which is normal of index three in the eigengroup, contains $\Gamma_0(9)$ as an index 4 subgroup, and the Fricke involution together with $x_3y_3$ and $x_3^2y_3^2$ are representatives of the nontrivial cosets.  Armed with the above information, we can construct a table whose entries are the components of $\hat{F}$.  Rows and columns are indexed from 0 to 8 (mod 9).

\[ \begin{array}{ccccccccc}
J & T_g(\tau) & T_g(\tau) & J & T_g(\tau) & T_g(\tau) & J & T_g(\tau) & T_g(\tau) \\
T_g(\frac{-1}{\tau}) & T_g(\frac{-1}{\tau+1}) & T_g(\frac{-1}{\tau+2}) & T_g(\frac{-1}{\tau+3}) & T_g(\frac{-1}{\tau+4}) & T_g(\frac{-1}{\tau+5}) & T_g(\frac{-1}{\tau+6}) & T_g(\frac{-1}{\tau+7}) & T_g(\frac{-1}{\tau+8}) \\
T_g(\frac{-1}{\tau}) & T_g(\frac{-1}{\tau+5}) & T_g(\frac{-1}{\tau+1}) & T_g(\frac{-1}{\tau+6}) & T_g(\frac{-1}{\tau+2}) & T_g(\frac{-1}{\tau+7}) & T_g(\frac{-1}{\tau+3}) & T_g(\frac{-1}{\tau+8}) & T_g(\frac{-1}{\tau+4}) \\
J & \zeta_3 T_g(\tau) & \zeta_3^2 T_g(\tau) & J & \zeta_3 T_g(\tau) & \zeta_3^2 T_g(\tau) & J & \zeta_3 T_g(\tau) & \zeta_3^2 T_g(\tau) \\
T_g(\frac{-1}{\tau}) & T_g(\frac{-1}{\tau+7}) & T_g(\frac{-1}{\tau+5}) & T_g(\frac{-1}{\tau+3}) & T_g(\frac{-1}{\tau+1}) & T_g(\frac{-1}{\tau+8}) & T_g(\frac{-1}{\tau+6}) & T_g(\frac{-1}{\tau+4}) & T_g(\frac{-1}{\tau+2}) \\
T_g(\frac{-1}{\tau}) & T_g(\frac{-1}{\tau+2}) & T_g(\frac{-1}{\tau+4}) & T_g(\frac{-1}{\tau+6}) & T_g(\frac{-1}{\tau+8}) & T_g(\frac{-1}{\tau+1}) & T_g(\frac{-1}{\tau+3}) & T_g(\frac{-1}{\tau+5}) & T_g(\frac{-1}{\tau+7}) \\
J & \zeta_3^2 T_g(\tau) & \zeta_3 T_g(\tau) & J & \zeta_3^2 T_g(\tau) & \zeta_3 T_g(\tau) & J & \zeta_3^2 T_g(\tau) & \zeta_3 T_g(\tau) \\
T_g(\frac{-1}{\tau}) & T_g(\frac{-1}{\tau+4}) & T_g(\frac{-1}{\tau+8}) & T_g(\frac{-1}{\tau+3}) & T_g(\frac{-1}{\tau+7}) & T_g(\frac{-1}{\tau+2}) & T_g(\frac{-1}{\tau+6}) & T_g(\frac{-1}{\tau+1}) & T_g(\frac{-1}{\tau+5}) \\
T_g(\frac{-1}{\tau}) & T_g(\frac{-1}{\tau+8}) & T_g(\frac{-1}{\tau+7}) & T_g(\frac{-1}{\tau+6}) & T_g(\frac{-1}{\tau+5}) & T_g(\frac{-1}{\tau+4}) & T_g(\frac{-1}{\tau+3}) & T_g(\frac{-1}{\tau+2}) & T_g(\frac{-1}{\tau+1}) 
\end{array} \]

The function $J$ above denotes the $SL_2(\Z)$-invariant function $T_1(\tau) = j(\tau)-744 = q^{-1} +  196884q + 21493760q^2 + \dots$.  Taking the Fourier transform, we can describe the vector-valued function $F$ by another table (with the same indexing).

\[ \begin{array}{ccccccccc}
f^0 & 0 & 0 & f^3 & 0 & 0 & f^6 & 0 & 0 \\ 
0 & 0 & f^4 & 0 & 0 & f^7 & 0 & 0 & f^1 \\
0 & f^5 & 0 & 0 & f^8 & 0 & 0 & f^2 & 0 \\
f^6 & 0 & 0 & f^0 & 0 & 0 & f^3 & 0 & 0 \\
0 & 0 & f^1 & 0 & 0 & f^4 & 0 & 0 & f^7 \\
0 & f^2 & 0 & 0 & f^5 & 0 & 0 & f^8 & 0 \\
f^3 & 0 & 0 & f^6 & 0 & 0 & f^0 & 0 & 0 \\
0 & 0 & f^7 & 0 & 0 & f^1 & 0 & 0 & f^4 \\
0 & f^8 & 0 & 0 & f^2 & 0 & 0 & f^5 & 0 
\end{array} \]

The functions are given by:

\[ \begin{aligned}
f^{0} &= (J + 2T_g)/3 = q^{-1} + 65628q + 7164752q^2 + \dots \\
f^{1} &= q^{-1/9} + 34752q^{8/9} + 4530744q^{17/9} + \dots \\
f^{2} &= 4124q^{5/9} + 1057504q^{14/9} + \dots \\
f^{3} &= (J - T_g)/3 = 65628q + 7164504q^2 + \dots \\
f^{4} &= 248q^{2/9} + 213126q^{11/9} + \dots \\
f^{5} &= 248q^{2/9} + 213126q^{11/9} + \dots \\
f^{6} &= (J - T_g)/3 = 65628q + 7164504q^2 + \dots \\
f^{7} &= 4124q^{5/9} + 1057504q^{14/9} + \dots \\
f^{8} &= q^{-1/9} + 34752q^{8/9} + 4530744q^{17/9} + \dots
\end{aligned} \]

The singular terms are $f^0$ with a pole of order $1$, and $f^1 = f^8$, with poles of order $1/9$.  Since the residues are integers, $F$ satisfies the second condition of Theorem \ref{thm:prods}.  We can describe the zeros of $\Psi$ by lifting the vector-valued function above to assign functions to each vector in $M^\vee$.  Recall from section \ref{sec:vector} that we identify $M = \Z \times 9\Z \times \Z \times \Z$ and $M^\vee = \frac{1}{9}\Z \times \Z \times \Z \times \Z$.  $\Psi$ has zeros of order one along rational quadratic divisors $\lambda^\perp$ for those $\lambda \in M$ of the following form:  Any $\lambda \in M$ with $Q(\lambda) = -1$ is assigned $f^0$, and has a pole corresponding to divisors $\sigma = A \tau$ for $A \in \Gamma_0(9)$.  Any $\lambda \in M$ with $Q(\lambda) = -9$, satisfying $\lambda/9 \in M^\vee$ with $\lambda/9$ assigned $f^1$ or $f^8$ has a pole corresponding to divisors $\sigma = A\tau$ for $A \in \begin{pmatrix} 0 & -1 \\ 9 & 0 \end{pmatrix} \Gamma_0(9)$.  Any $\lambda \in M$ with $Q(\lambda) = -9$, satisfying $\lambda/3 \in M^\vee$ with $\lambda/3$ assigned $f^0$ has a pole corresponding to divisors $\sigma = A\tau$ for $A \in x_3y_3 \Gamma_0(9)$ (if $\lambda/3 \equiv (1/3,3,0,0)$ mod $M$) or $A \in x_3^2 y_3^2 \Gamma_0(9)$ (if $\lambda/3 \equiv (2/3,6,0,0)$ mod $M$).  $\Psi$ also has a product expansion at the cusp $(0,0,0,1) = (i\infty,0)$ given by:
\[p^{-1} \prod_{m > 0, n \in \frac{1}{9}\Z} (1 - p^m q^n)^{c(m,n)} \]
Here, $c(m,n)$ is given by $c_1(9mn) + c_2(mn)/3$ if $m/3$ and $3n$ are integers, and $c_1(9mn)$ otherwise, where $T_g(\tau) = \sum_{n \in Z} c_1(n) q^n$ and $\sum_{n \in Z} c_2(n) q^n = J(\tau) - T_g(\tau) = 196884q + 21493512q^2 + 864299970q^3 + 20245856256q^4 + \dots$. 

This product is the expansion of $T_g(\sigma) - T_g(-1/\tau)$.  Up to a change of grading, this product is equivalent to the Weyl-Kac-Borcherds denominator of one of the monstrous Lie algebras described in section 10 of \cite{B92}.

There is a good reason why most of the entries of $F$ are zero, and it is because $F$ descends to a vector-valued modular function for a lattice $L \times I\! I_{1,1}$ that contains $M = I\! I_{1,1}(9) \times I\! I_{1,1}$ as an index three sublattice, where $L$ is isomorphic to $\Z \times \Z$ with the quadratic form $(a,b) \mapsto 3ab + b^2$.  In particular, $F$ is supported on $L^\vee/M$, which is an index three subgroup of $M^\vee/M$.  Other elements whose McKay-Thompson series have level larger than their order (i.e., with $h>1$ in the sense of \cite{CN79}) behave similarly.

For most elements, namely those with $h=1$, the lattice quotients have the form $\Z/N\Z \times \Z/N\Z$.  However, for 3C, $L^\vee/L \cong \Z/9\Z$, and we expect that the category of representations of the fixed point vertex algebra $(V^\natural)^g$ has a braided tensor structure, with fusion algebra naturally isomorphic to $\Z[L^\vee/L]$, such that each $f^i$ equal to the graded dimension of the corresponding irreducible object.  Since it is in general quite difficult to determine the representation theory of fixed point algebras, this interpretation is still conjectural.

We can view the vector-valued function $F$ in a more holistic way, as a vector-valued character of an abelian intertwining algebra, in the sense of \cite{DL93}, and the theory of abelian intertwining algebras will still work if we replace $\langle g \rangle$ with any cyclic subgroup of $\MM$, or with a choice of many non-cyclic abelian subgroups.  There is an alternative viewpoint that extends to larger subgroups, given by vertex algebra objects (i.e., singular commutative rings) in a braided tensor category.   
We intend to show in forthcoming work that these objects are much easier to study than representations of fixed point algebras, once some conceptual hurdles are overcome.  For the 3C case, there is more than one interesting way to view the braided category in question:
\begin{enumerate}
\item The most straightforward view is that this is the category of $L^\vee/L$-graded vector spaces, where the braided monoidal structure is classified up to equivalence by the induced $\Q/\Z$-valued quadratic form.  This view can be generalized to abelian subgroups.
\item To connect this structure to the action of $\langle g \rangle$, one can view it as the Drinfeld center of the monoidal category of $\Z/3\Z$-graded vector spaces, with associator twisted by a nontrivial 3-cocycle with coefficients in $\C^\times$.  Equivalently, the braided category is the category of modules of a twisted quantum double of $\C[\Z/3\Z]$ \cite{DPR90}.  This viewpoint extends to treat nonabelian subgroups of the monster.
\item The abelian category of $\Z/3\Z$-graded vector spaces can be viewed more naturally as vector bundles on the group $\Z/3\Z$, and the twisted associator on the convolution tensor structure describes a nontrivial 2-group structure, equivalent to the based loop space of a nontrivial $K(\C^\times,2)$-torsor over $B\Z/3\Z$.  Our monoidal category is given by sheaves on this 2-group with convolution, and its center is given by either conjugation-equivariant sheaves on the 2-group, or equivalently, sheaves on the free loop space of the torsor.  If we expand our view to the abelian category $\mathcal{C}$ of all twisted modules of $V^\natural$, there is a natural strict action of a 2-group $\widetilde{\MM}$ whose structure is described up to equivalence (following Sinh's conditions - see section 8 of \cite{BL04}) by $\pi_0(\widetilde{\MM}) \cong \MM$, $\pi_1(\widetilde{\MM}) \cong \Aut(Id_\mathcal{C})$, and a nontrivial element of $H^3(\pi_0(\widetilde{\MM}), \pi_1(\widetilde{\MM}))$, where the module structure is trivial.
\end{enumerate}

The constant ambiguity in the generalized moonshine conjecture seems to arise from the nontriviality of the higher structure in $\widetilde{\MM}$, and we can see it manifested in the 3C case.  There is an action of the centralizer $C_{\widetilde{\MM}}(g)$ on the category of $g$-twisted modules, which induces an action of a $\C^\times$ central extension of $C_\MM(g)$ on any fixed irreducible $g$-twisted module, with a lift of $g$ acting by the conformal operator $e(L_0)$.  The class of the central extension is given by restriction of the loop transgression of a 3-cocycle on $\MM$ (see section 3 of \cite{W08}).  Further restriction of the 2-cocycle to the cyclic group generated by $g$ is yields a trivializable cocycle, but any trivialization is incompatible with the assignment of $e(L_0)$ as a distinguished lift of $g$.  When we pull back to the degree 3 central extension $\Z/9\Z$, the compatibility is achievable.

In general, when we deal with cyclic subgroups of $\MM$, we can trivialize the restricted 2-group structure by passing to an abelian central extension, and in the above treatment of the 3C case, we did this by essentially pretending that $g$ had order 9.  Since $Th$ has no nontrivial discrete central extensions, we can descend to an action of $\Z/9\Z \times Th$ on the irreducible twisted module. Geometrically, the nontriviality of the 3-cocycle implies that $g$-twisted modules are not canonically attached to ramified points on an $\MM$-cover of an algebraic curve, because the twisted modules are not admissible in the sense of \cite{FS04}.  Instead, one can either introduce a formalism of $\widetilde{\MM}$-covers, or attach modules to nonzero one-jets at those points.  The $e(L_0)$ incompatibility mentioned in the previous paragraph implies that the dependence on choice of 1-jet is nontrivial.  We will treat the geometric formalism more thoroughly in \cite{C}.

\section{Lie algebras}

In this section, we give a conditional proof of the generalized moonshine conjecture.  It requires strong assumptions concerning group actions on certain Lie algebras, together with a technical hypothesis concerning genus 1 correlation functions.

In the first part, we construct Lie algebras whose denominator identities give the product formulas from the previous section.  Then we introduce a notion of non-Fricke compatibility for Lie algebras, and show that the Lie algebras we constructed using McKay-Thompson series have positive subalgebras that are Fricke or non-Fricke compatible with cyclic group actions.  We introduce some assumptions for groups acting on these Lie algebras, and show that under those assumptions, the resulting characters are weakly Hecke-monic.  The weakly Hecke-monic functions are automatically genus zero in the Fricke case, but otherwise, we need an additional assumption.

\subsection{Denominator formulas}

Generalized Kac-Moody algebras were introduced by Borcherds in \cite{B88}.  The theory is very similar to that of Kac-Moody algebras, since the algebras are described by a Cartan matrix giving generators and relations, with the main distinction being the possibility of having imaginary simple roots, i.e., generators of non-positive norm. 
Like Kac-Moody algebras, they admit character formulas, but the formulas have correction terms for the imaginary simple roots.  In the case of the trivial representation, we obtain the Weyl-Kac-Borcherds denominator formula:
\[ \sum_{w \in W} \epsilon(w) w(S) = e(\rho) \prod_{\alpha \in \Delta^+} (1- e(-\alpha))^{\mathrm{mult}(\alpha)} \]
Here, $S = e(\rho) \sum_s \epsilon(s) e(s)$, where $s$ runs over all sums of pairwise orthogonal imaginary simple roots.  From a homological algebra standpoint, this formula expresses an isomorphism of virtual vector spaces $H^*(E) \cong \wedge^*(E)$ arising from the Chevalley resolution of the trivial representation.

We are interested in the rank two setting, where the root lattice is Lorentzian of signature $(1,1)$.  Since the real simple roots generate a Kac-Moody algebra, there are at most two of them, but the imaginary simple roots can have arbitrary multiplicity.  For example, one can generate a large family of such algebras by choosing an arbitrary non-negative Laurent series of the form $f(q) = q^{-1} + \sum_{n>0} a_n q^n \in q^{-1}\N[[q]]$, setting a real simple root at $(1,-1)$, and imaginary simple roots of multiplicity $a_n$ at $(1,n)$.  This algebra has a Weyl group of order two, and the Weyl-Kac-Borcherds denominator formula equates $f(p) - f(q)$ with $p^{-1}$ times a product of binomials of the form $1-p^mq^n$.

\begin{lem} \label{lem:cartan} 
Suppose we are given a product formula of one of the following two types:
\begin{enumerate}
\item $f(p) - f(q^{1/N}) = p^{-1} \prod_{m >0, n \in \frac{1}{N}\Z} (1-p^m q^n)^{c(m,n)}$, where $f(q) = q^{-1} + \sum_{n=1}^\infty c(1,n/N) q^n \in q^{-1}\N[[q]]$.
\item $f(p) - g(q^{1/N}) = p^{-1} \prod_{m >0, n \in \frac{1}{N}\Z} (1-p^m q^n)^{c(m,n)}$, where $f(q) = q^{-1} \prod_{n=1}^\infty (1-q^n)^{c(n,0)}$ satisfies $c(n,0) \in \N$ for all $n$, and $g(q) = \sum_{n \geq 0} c(1,n) q^n \in \N[[q]]$ 
\end{enumerate}
Then there exists a unique generalized Kac-Moody algebra with that product formula as denominator formula.
\end{lem}
\begin{proof}
This proof follows a strategy similar to the construction of the monster Lie algebra in \cite{JLW95}.

We begin by choosing the root lattice, whose underlying group is $\Z \times \frac{1}{N}\Z$, and whose rational bilinear form is induced from the rational quadratic form $Q((m,n)) = mn$.  In particular, the inner product is given by $\langle (a,b/N), (c,d/N) \rangle = \frac{ad+bc}{N}$ for all integers $a,b,c,d$.  We also introduce notation for describing the characters of vector spaces graded by the root lattice: A one dimensional vector space concentrated in degree $(m,n) \in \Z \times \frac{1}{N}\Z$ has character $p^mq^n$.

For the first type of product formula, we construct a Cartan matrix by allocating $c(1,n)$ rows and columns to degree $(1,n)$, and placing $-Nm-Nn$ in every entry whose row has degree $(1,n)$ and whose column has degree $(1,m)$, for each $m, n \in \frac{1}{N}\Z$.  The resulting matrix has the following properties:
\begin{enumerate}
\item The matrix is symmetric.
\item There is exactly one non-negative diagonal entry.  It is equal to 2, and corresponds to $c(1,-1/N) = 1$.
\item Off-diagonal entries are non-positive integers.
\end{enumerate}
The three conditions for a generalized Cartan matrix listed in \cite{B95a} are therefore satisfied by our matrix.  We form the universal generalized Kac-Moody algebra described by this matrix, and let $\mathfrak{g}$ be the quotient by the center.  Since the matrix has one positive diagonal entry, there is one real simple root, and the Weyl group has order 2.  The Cartan subalgebra has dimension 2, and we may choose the Weyl vector to be $\rho = (-1,0)$, since that vector satisfies the defining property that $(\rho, \alpha_i) = -\frac{1}{2}(\alpha_i,\alpha_i)$ for all simple roots $\alpha_i$.  If $f$ has nonzero terms in positive degree, then the Weyl vector is unique.

The left side of the denominator formula for $\mathfrak{g}$ has the form $\sum_{w \in W} \epsilon(w) w(S)$.   Since all imaginary simple roots have negative norm, all sets of pairwise orthogonal imaginary simple roots have size at most one, so
\[ S = e(\rho) \sum_s \epsilon(s) e(s) = p^{-1} \sum_{n \in \frac{1}{N}\Z} c(1,n) pq^n. \]
The left side of the denominator formula is therefore $\sum_{n} c(1,n) q^n - \sum_{n} c(1,Nn) p^{Nn}$, which is the left side of the product formula.  The right side of the denominator formula for $\mathfrak{g}$ has the form 
\[ e(\rho) \prod_{\alpha \in \Delta^+} (1- e(-\alpha))^{\text{mult}(\alpha)} \]
The positive roots $\alpha \in \Delta^+$ have the property that $e(\alpha) = p^m q^n$ for some $m>0$ and $n \in \frac{1}{N}\Z$, so we can rewrite this expression as
\[ p^{-1} \prod_{m>0, n \in \frac{1}{N}\Z} (1-p^m q^n)^{\text{mult}(\alpha)}. \]
Since $f(p)-f(q^{1/N})$ has a unique product expansion in binomials of the form $(1-p^aq^b)$, the multiplicities of the roots $\alpha$ are given by $c(m,n)$.

For the second type, we construct a Cartan matrix by allocating $c(1,n)$ rows and columns to degree $(1,n)$, $c(m,0)$ rows and columns to degree $(m,0)$, and specifying the following entries:
\begin{enumerate}
\item We place the integer $-Nm-Nn$ in every place with row degree $(1,n)$ and column degree $(1,m)$.
\item We place the integer $-Nmn$ in every place with row degree $(1,n)$ and column degree $(m,0)$, or with row degree $(m,0)$ and column degree $(1,n)$.
\item We place $0$ in every place with row degree $(m,0)$ and column degree $(n,0)$.
\end{enumerate}
This matrix is symmetric, with non-positive off-diagonal entries, and it has no positive diagonals, so it satisfies the criteria of \cite{B95a}.  As before, we form the universal generalized Kac-Moody algebra described by this matrix, and let $\mathfrak{g}$ be the quotient by its center.  This algebra has no real simple roots, so its Weyl group is trivial.  As in the first case, the Cartan subalgebra is two-dimensional and the Weyl vector $\rho$ is given by $(-1,0)$.

Since the Weyl group of this algebra is trivial, the left side of the denominator formula for $\mathfrak{g}$ is just $S = e(\rho) \sum_s \epsilon(s) e(s)$.   Unlike the first case, there are imaginary simple roots of norm zero, and they are orthogonal to each other, so the sum is somewhat more complicated.  We note that the imaginary simple roots of degree $(1,n)$ for $n>0$ are not orthogonal to any others, so they appear in the sum once.  The rest of the imaginary simple roots have degree $(m,0)$ for positive integers $m$, and these roots are mutually orthogonal.  In other words, $S$ is given by adding $\sum_{n =1}^\infty -c(1,n/N)pq^{n/N}$ to an alternating sum of over all finite sets of simple roots of degree $(m,0)$.  The alternating sum can be viewed as the character of the exterior algebra whose basis is given by those simple roots, and this has the form of an infinite product.  We get:

\[ \begin{aligned}
\sum_s \epsilon(s) e(s) &= \sum_{n =1}^\infty -c(1,n/N)pq^{n/N} + \prod_{m=1}^\infty (1-p^m)^{c(m,0)} \\
&= p(f(p) - g(q^{1/N}))
\end{aligned} \]

Therefore, $S  = f(p) - g(q^{1/N})$.

The right side of the denominator formula is similar to the first case:
\[ p^{-1} \prod_{m>0, n \in \frac{1}{N}\Z} (1-p^m q^n)^{\text{mult}(\alpha)}. \]
However, in this case, there are no real simple roots, so $n$ can be restricted to range over positive numbers.  As before, uniqueness of the product formula implies the multiplicities are given by the exponents $c(m,n)$.
\end{proof}

\noindent\textbf{Remark:} If the exponents in a product are non-integral or negative, we can view the product as a denominator formula for a more general Lie algebra object in virtual vector spaces.  The monstrous Lie superalgebras in \cite{B92} correspond to the setting of supervector spaces and negative integer exponents.

There is a distinguished subclass of these generalized Kac-Moody algebras, made of those whose denominator formulas have modular forms on the left, and products with coefficients of modular forms on the right.  In contrast to the general case, it is expected that there are only finitely many examples of these.

\begin{thm} \label{thm:gkm} 
For each $g \in \MM$, the product expansion of $T_g(\sigma) - T_g(-1/\tau)$ as given in Proposition \ref{prop:prod2} is the denominator formula of a unique generalized Kac-Moody Lie algebra.
\end{thm}
\begin{proof}
We split this question into two cases:
\begin{enumerate}
\item $g$ is Fricke, i.e., if $T_g(\tau)$ is invariant under some $\tau \mapsto -1/N\tau$.
\item $g$ is non-Fricke, i.e., if $T_g(\tau)$ is regular at zero.
\end{enumerate}

For the Fricke case, $T_g(\sigma) - T_g(-1/\tau) = T_g(\sigma) - T_g(\tau/N)$, so if we assume the coefficients of $T_g(\tau)$ (equivalently, those of $T_g(-1/\tau)$) are nonnegative integers, then the expansion has the form of the first case of Lemma \ref{lem:cartan}.  The coefficients are already known to be integers, so it suffices to show that they are nonnegative.  By \cite{DLM00}, there exists an irreducible $g$-twisted module of $V^\natural$, unique up to isomorphism, such that its character is a constant times the $q$-expansion $T_g(-1/\tau)$.  Since the character has nonnegative integer coefficients and the leading term of $T_g(-1/\tau)$ is 1, this constant must be a positive rational number, and the coefficients of $T_g(\tau)$ are nonnegative.

For the non-Fricke case, the power series expansion of $T_g(\sigma) - T_g(-1/\tau)$ has no terms with any negative powers of $q$.  By the second case of Lemma \ref{lem:cartan}, it suffices to check that $T_g(-1/\tau)$ has nonnegative integer coefficients, and that $T_g(\tau)$ is an eta-product $\eta(a_1 \tau)^{b_1} \eta(a_2 \tau)^{b_2} \dots \eta(a_n \tau)^{b_n}$ whose exponents satisfy the following positivity condition:
\begin{itemize}
\item For any residue class $k$ modulo $\mathrm{lcm}(\{a_1, \dots, a_n \})$,
\[ \underset{a_i | k}{\sum_{i \in \{1, \dots, n\} }} b_i \geq 0. \]
Here, the notation $a_i|k$ means the order of $k$ in $\Z/\mathrm{lcm}(\{a_1, \dots, a_n \})\Z$ divides the order of $a_i$.
\end{itemize}
The information necessary for verifying these conditions already exists in encoded form in Tables 2 and 3 in \cite{CN79}.  Table 2 gives a ``symbol'' for each conjugacy class, and the non-Fricke elements are those whose symbol has the form $n-$, $n|h$ for $h \neq n$, $n + s_1, s_2, \dots$ where no $s_i$ is equal to $n$, and $n|h + s_1, s_2, \dots$, where no $s_i$ is equal to $n/h$.  Table 3 gives eta-product expansions of some of the non-Fricke elements, and the rest are described in the ``Harmonies and Symmetrizations'' column.  Checking this information in each case, we find that the positivity condition is satisfied, and that $T_g(-1/N\tau) = k_1 + k_2/T_g(\tau)$ for some positive integers $k_1$, $k_2$.  The coefficients of $T_g(-1/\tau)$ are therefore integers, and either by using the results of \cite{DLM00} as in the Fricke case, or by analyzing the eta product expansion of $\frac{T_g(-1/\tau)-k_1}{k_2}$, we find that the coefficients of $T_g(-1/\tau)$ are non-negative.
\end{proof}

\subsection{Compatibility}

Let $G$ be a group, and let $g$ be an element of order $N$ in the center of $G$.  Suppose we have a collection $\mathcal{V} = \{ V^{i,j/N}_k | i, j \in \Z/N\Z, k \in \frac{1}{N}\Z \}$ of $G$-modules, such that the action of $g$ on $V^{i,j/N}_k$ is given by constant multiplication by the root of unity $e(j/N)$, and $\dim V^{i,j/N}_k$ grows subexponentially with $k$, i.e., for any $\epsilon>0$, there is some $C>0$ such that for all $i,j,k$, $\dim V^{i,j/N}_k < Ce^{\epsilon k}$.

From \cite{C10}, we say that a complex Lie algebra $E$ is Fricke compatible with the triple $(G,g,\mathcal{V})$ if:
\begin{enumerate}
\item $E$ is graded by $\Z_{> 0} \times \frac{1}{N}\Z$, with finite dimensional homogeneous components $E_{i,j}$.  We write $E = \bigoplus_{i>0, j \in \frac{1}{N}\Z} E_{i,j} p^i q^j$, where $p$ is a $(1,0)$ degree shift, and $q$ is a $(0,1)$ degree shift.  We can view this as a character decomposition under an action of a two dimensional torus $H$.
\item $E$ admits a homogeneous action of $G$ by Lie algebra automorphisms, such that we have $G$-module isomorphisms $E_{i,j} \cong V^{i,j}_{1+ij}$
\item The homology of $E$ is given by:
\begin{itemize}
\item $H_0(E) = \C$
\item $H_1(E) = \bigoplus_{n \in \frac{1}{N}\Z} V^{1,n}_{1+n} pq^n$
\item $H_2(E) = p \bigoplus_{m=1}^\infty V^{1,-1/N}_{1-1/N} \otimes V^{m,1/N}_{1+m/N} p^m$
\item $H_i(E) = 0$ for $i>2$.
\end{itemize}
\item $E_{1,-1/N} \cong V^{1,-1/N}_{1-1/N}$ is one dimensional.
\end{enumerate}

We define here the complementary notion that is suited to non-Fricke elements of the monster.

\begin{defn}
A complex Lie algebra $E$ is non-Fricke compatible with the triple $(G,g,\mathcal{V})$ if:
\begin{enumerate}
\item $E$ is graded by $\Z_{> 0} \times \frac{1}{N}\N$, with finite dimensional homogeneous components $E_{i,j}$.  We write $E = \bigoplus_{i \in \N, j \in \frac1N \N} E_{i,j} p^i q^j$, where $p$ is a $(1,0)$ degree shift, and $q$ is a $(0,1)$ degree shift.  We can view this as a character decomposition under an action of a two dimensional torus $H$.
\item $E$ admits a homogeneous action of $G$ by Lie algebra automorphisms, such that we have $G$-module isomorphisms $E_{i,j} \cong V^{i,j}_{1+ij}$
\item The homology of $E$ is given by:
\begin{itemize}
\item $H_0(E) = \C$.
\item $H_1(E) = \bigoplus_{n \in \frac{1}{N}\N} V^{1,n}_{1+n} pq^n \oplus \bigoplus_{m =1}^\infty V^{m,0}_1 p^m$.
\item $H_i(E) = \bigwedge\!^i (\bigoplus_{m=1}^\infty V^{m,0}_1 p^m)$ for $i \geq 2$.
\end{itemize}
\end{enumerate}
\end{defn}

\noindent\textbf{Remark:} The last condition matches the higher homology of $E$ with the higher homology of the abelian subalgebra $\bigoplus_{m>0} V^{m,0}_1 p^m$.

\begin{prop} \label{prop:compatible} 
Let $G = \langle g \rangle$ be a cyclic group of order $N$, and let $\phi: G \to \MM$ be a homomorphism.  Let $T_g(\tau)$ denote the McKay-Thompson series for the automorphism of $V^\natural$ described by $\phi(g)$, and assume it is invariant under $\Gamma_0(N)$.  Let $\mathcal{V}$ be the collection of $G$-modules defined by setting the dimension of $V^{i,j/N}_k$ to be the coefficient $c_{i,j}(k)$ of the vector-valued form associated to $T_g$, with $g$ acting by multiplication by $e(j/N)$.  Let $W_g$ be the generalized Kac-Moody Lie algebra whose denominator formula is the product formula for the expansion of $T_g(\sigma) - T_g(-1/\tau)$, given in Proposition \ref{prop:prod2}.  Define an action of $G$ on $W_g$ by setting $g$ to act on vectors of degree $(m,n)$ by $e(n)$, and let $E_g$ be the Lie subalgebra of $W_g$ formed by the positive root spaces.  If $g$ is Fricke, then $E_g$ is Fricke compatible with $(G,g,\mathcal{V})$, and if $g$ is non-Fricke, then $E_g$ is non-Fricke compatible with $(G,g,\mathcal{V})$.
\end{prop}

\begin{proof}
We split the problem into two cases:

\noindent\textbf{First:} Suppose $g$ is Fricke.  From Theorem \ref{thm:gkm}, there is one real simple root in degree $(1,-1/N)$, and rest of the positive roots are imaginary with multiplicity $c_{n,m}(-mn)$ in degree $(m,n)$ for all $m \in \N$, $n \in \frac{1}{N}\N$.  $E_g$ therefore satisfies the grading restriction condition, and $E_{1,-1/N}$ is one dimensional.

$G$ acts by Lie algebra automorphisms, because the action of $g$ is determined by grading in a multiplicative way, i.e, the action factors through the adjoint action of the torus $H$.

The action of $G$ on root spaces of $E_g$ is identified with the action on the spaces in $\mathcal{V}$, since $g$ acts by $e(n)$ on $V^{m,n}_{1+mn}$ for all $m \in \Z, n \in \frac{1}{N}\Z$.

Since $W_g$ is generalized Kac-Moody, the homology of $E_g$ is given by a formula analogous to the denominator formula.  The homology groups can be calculated by the rule at the end of section 4 in \cite{B92}, i.e., noting that Proposition 7.9 in \cite{GL76} can be adapted to the case of generalized Kac-Moody algebras admitting a Weyl vector.  This adaptation was executed in detail in the proof of Theorem 3.6 in \cite{J04}.  We find that $H_i(E_g)$ is the subspace of $\bigwedge^i(E_g)$ spanned by homogeneous vectors of degree $r$ satisfying $(r+\rho)^2 = \rho^2$, or equivalently, $(r,r + 2\rho) = 0$.  $\rho$ is the vector $(-1,0)$, so the only degrees $r$ for which the condition holds have the form $(m,0)$ or $(1,n)$ for $m \geq 0$, $n \in \frac{1}{N}\Z$.  Since $E_g$ is supported in degrees $(a,b/N)$ for $a$ a positive integer and $b \in \Z$, the only nonzero vectors of $\bigwedge(E_g)$ in degree $(1,n)$ are the simple roots in $E_g = \bigwedge^1(E_g)$.  Since $E_{1,-1/N}$ is one dimensional, the only nonzero vectors of $\bigwedge(E_g)$ in degree $(m+1,0)$ lie in $V^{m,1/N}_{1+m/N} \wedge V^{1,-1/N}_{1-1/N} \subset \bigwedge^2(E_g)$.  Therefore the homology has the necessary form for Fricke-compatibility.

\noindent\textbf{Second:} Suppose $g$ is non-Fricke.  From Theorem \ref{thm:gkm}, there are no real simple roots, and the positive roots are all imaginary with multiplicity $c_{n,m}(-mn)$ in degree $(m,n)$ for all $m \in \N$, $n \in \frac{1}{N}\N$.  $E_g$ therefore satisfies the grading restriction condition.

By the same argument as in the first case, $G$ acts by Lie algebra automorphisms, and there are G-module isomorphisms $E_{i,j} \cong V^{i,j}_{1+ij}$.

As in the Fricke case, $H_i(E_g)$ is the subspace of $\bigwedge^i(E_g)$ spanned by homogeneous vectors of degree $r$ satisfying $(r,r+2\rho) = 0$, and the only degrees $r$ for which this condition holds have the form $(m,0)$ and $(1,n)$ for $m>0$, $n \in \frac{1}{N}\N$.  Since $E_g$ is supported in degrees $(a,b/N)$ for $a$ a positive integer and $b \in \N$, the only nonzero vectors of $\bigwedge(E_g)$ in degree $(1,n)$ are the degree $(1,n)$ vectors in $E_g = \bigwedge^1(E_g)$, and the only nonzero vectors in degree $(m,0)$ are elements in the exterior algebra of $\bigoplus_{m=1}^\infty V^{m,0}_1 p^m$.  This has the form necessary to satisfy the non-Fricke condition.
\end{proof}

\begin{cor} \label{cor:zat0} 
Under the hypotheses of Proposition \ref{prop:compatible}, $T_g(-1/\tau) = \sum_{n \in \frac{1}{N}\Z} \dim V^{1,n}_{1+n} q^n$.
\end{cor}
\begin{proof}
By Proposition \ref{prop:compatible}, the part of the homology of $E_g$ concentrated in degree $(1,n)$ is $V^{1,n}_{1+n} \subset \bigwedge^1(E_g)$, in $H_1$.  By Theorem \ref{thm:gkm}, the denominator identity for $W_g$ is given by $T_g(\sigma) - T_g(-1/\tau) = p^{-1} \prod_{m>0,n \in \frac{1}{N}\Z} (1-p^m q^n)^{c_{m,Nn}(mn)}$.  The pure-in-$q$ part of the left side is $T_g(-1/\tau)$, while the pure-in-$q$ part of the right side is given by $\sum_{n \in \frac{1}{N}\Z} c_{1,Nn}(n) q^n$, which is equal to $\sum_{n \in \frac{1}{N}\Z} \dim E_{1,n} q^n$.  We find that:
\[ \begin{aligned}
T_g(-1/\tau) &= \sum_{n \in \frac{1}{N}\Z} c_{1,Nn}(n) q^n \\
&= \sum_{n \in \frac{1}{N}\Z} \dim E_{1,n} q^n \\
&= \sum_{n \in \frac{1}{N}\Z} \dim V^{1,n}_{1+n} q^n
\end{aligned} \]
\end{proof}

\subsection{Twisted denominator formulas}

In this section, $G$ is a finite group with an element $g$ of order $N$ in its center, and $\mathcal{V}$ is a collection of $G$-modules as described in the previous section.

Following \cite{C10}, we can define ``formal orbifold partition functions'' for any $h \in G$:
\[ Z(g^k,g^lh^m,\tau) := \sum_{n \in \frac{1}{N}\Z} \underset{n \in kr+\Z}{\sum_{r \in \frac{1}{N}\Z/\Z}} \mathrm{Tr}(g^lh^m|V^{k,r}_{1+n}) e(n\tau) \]
These are a priori formal power series in $q^{1/N}$, but they converge on $\HH$, by the subexponential growth condition in the definition of $\mathcal{V}$.  We refer to the collection of these functions as $Z$.

\begin{lem} 
Suppose $G = \langle g \rangle$ is a cyclic group of order $N$ equipped with a homomorphism $\pi: G \to \MM$, such that $T_g(\tau)$ has level $N$.  Let $\mathcal{V}$ be the compatible family given in Proposition \ref{prop:compatible}.  Then for all $i \in \Z$, the orbifold partition function $Z(g,g^i,\tau)$ is equal to $T_g(\frac{-1}{\tau+i})$.
\end{lem} 
\begin{proof}
By Corollary \ref{cor:zat0},
\[ \begin{aligned}
Z(g,1,\tau) &= \sum_{n \in \frac{1}{N}\Z} \mathrm{Tr}(1|V^{1,n}_{1+n}) e(n\tau) \\
&= T_g(-1/\tau)
\end{aligned} \]
If we substitute $\tau+i$ for $\tau$, and note that $g^i$ acts on $V^{1,n}_{1+n}$ by multiplication by $e(in)$, we get:
\[ \begin{aligned}
T_g(\frac{-1}{\tau+i}) &= \sum_{n \in \frac{1}{N}\Z} \mathrm{Tr}(1|V^{1,n}_{1+n}) e(n\tau + in) \\
&= \sum_{n \in \frac{1}{N}\Z} \mathrm{Tr}(g^i|V^{1,n}_{1+n}) e(n\tau) \\
&= Z(g,g^i,\tau)
\end{aligned} \]
\end{proof}


Recall from \cite{C10} Proposition 6.1, that for any Lie algebra $E$ Fricke-compatible with the data $(G,g,\mathcal{V})$, and any $h \in G$, we had:
\[ \begin{aligned}
p^{-1} &+ \sum_{m > 0} \mathrm{Tr}(h|V^{1,-1/N}_{1-1/N}) \mathrm{Tr}(h|V^{m,1/N}_{1+m/N}) p^m - \sum_{n \in \frac{1}{N}\Z} \mathrm{Tr}(h|V^{1,n}_{n+1}) q^n \\
&= p^{-1} \exp \left( - \sum_{i > 0} \sum_{m > 0, n \in \frac{1}{N}\Z} \mathrm{Tr}(h^i | V^{m,n}_{1+mn})p^{im}q^{in}/i \right)
\end{aligned} \]

\begin{prop} 
Suppose $G$ and $g$ are as described in the previous section, and suppose $E$ is non-Fricke compatible with $(G,g,\mathcal{V})$.  Then for any $h \in G$,
\[ \begin{aligned}
p^{-1} &\exp \left( - \sum_{i>0} \sum_{m>0} \mathrm{Tr}(h^i|V^{m,0}_1) p^{im}/i \right) - \sum_{n \in \frac{1}{N}\N} \mathrm{Tr}(h|V^{1,n}_{1+n}) q^n = \\
& = p^{-1} \exp \left( - \sum_{i > 0} \sum_{m > 0, n \in \frac{1}{N}\N} \mathrm{Tr}(h^i | V^{m,n}_{1+mn})p^{im}q^{in}/i \right) .
\end{aligned} \]
\end{prop}
\begin{proof}
It suffices to show that the left and right sides of the equation describe $\mathrm{Tr}(h|H_\bullet(E))$ and $\mathrm{Tr}(h|\bigwedge^\bullet(E))$, respectively.  The proof is essentially identical to \cite{C10}, Proposition 6.1.  The exponential portions are a formal consequence of the definition of Adams operations.
\end{proof}


Recall from \cite{C10} that we define equivariant Hecke operators by
\[ T_n Z(g,h,\tau) = \frac{1}{n} \sum_{ad=n, 0 \leq b < d} Z(g^d, g^{-b}h^a, \frac{a\tau+b}{d}). \]
If we apply the substitution $g \mapsto g^i, h \mapsto g^j$, we get the formula in the beginning of  Section \ref{sec:hecke-monic}.  A function $Z$ on $\Hom(\Z \times \Z, G)/G \underset{\pm \Z}{\times} \HH$ is weakly Hecke-monic for $(g,h)$ if $T_nZ(g,h,\tau)$ is a monic polynomial of degree $n$ in $Z(g,h,\tau)$ for all $n >0$.  When $g$ is Fricke, the combination of Proposition \ref{prop:compatible} in this paper and \cite{C10}, Proposition 6.2 directly imply $Z(g,h,\tau)$ is weakly Hecke-monic.

\begin{prop} \label{prop:monic} 
Let $\phi: \langle g \rangle \to \MM$ be a homomorphism, such that $T_g(\tau)$ is invariant under $\Gamma_0(|g|)$ and $\phi(g)$ is a non-Fricke element of $\MM$. Let $G$ be a group with $g$ in its center, and let $\mathcal{V}$ and $E$ be defined as $\langle g \rangle$-modules as in Proposition \ref{prop:compatible}.  Suppose that the actions of $\langle g \rangle$ on $\mathcal{V}$ and $E$ extend to actions of $G$, such that they are non-Fricke compatible.  Then for any $h \in G$, $Z$ is weakly Hecke-monic for $(g,h)$.
\end{prop}
\begin{proof}
This is very similar to Proposition 6.2 in \cite{C10}.  We multiply both sides of the twisted denominator formula and take logarithms:

\[ \begin{aligned}
\log &\left( \exp \left( - \sum_{i>0} \sum_{m>0} \mathrm{Tr}(h^i|V^{m,0}_1) p^{im}/i \right) - p\sum_{n \in \frac{1}{N}\N} \mathrm{Tr}(h|V^{1,n}_{1+n}) q^n \right) \\
&= - \sum_{i > 0} \sum_{m > 0, n \in \frac{1}{N}\N} \mathrm{Tr}(h^i | V^{m,n}_{1+mn})p^{im}q^{in}/i \\
&= - \sum_{m > 0, n \in \frac{1}{N}\N} \sum_{0<a|(m,Nn)} \frac{1}{a}\mathrm{Tr}(h^a | V^{m/a,n/a}_{1+mn/a^2}) p^m q^n \\
&= - \sum_{m > 0} \sum_{ad=m} \frac{1}{a}\sum_{n \in \frac{1}{N}\N} \mathrm{Tr}(h^a | V^{d,n}_{1+dn}) p^m q^{an} \\
&= - \sum_{m > 0} \sum_{ad=m} \frac{1}{a} \sum_{0 \leq b < d} \frac{1}{d} \sum_{n \in \frac{1}{N}\N} \underset{n \in dr+\Z}{\sum_{r \in \frac{1}{N}\Z/\Z}} e(-br) \mathrm{Tr}(h^a | V^{d,r}_{1+n}) e(br) q^{an/d} p^m \\
&= - \sum_{m > 0} \frac{1}{m} \underset{0 \leq b < d}{\sum_{ad=m}} \sum_{n \in \frac{1}{N}\N} \underset{n \in dr+\Z}{\sum_{r \in \frac{1}{N}\Z/\Z}} \mathrm{Tr} (g^{-b}h^a | V^{d,r}_{1+n}) e(n\frac{a\tau + b}{d}) p^m \\
&= -\sum_{m > 0} \frac{1}{m} \underset{0 \leq b < d}{\sum_{ad=m}} Z(g^d, g^{-b}h^a,\frac{a\tau+b}{d}) p^m \\
&= -\sum_{m > 0} T_m Z(g,h, \tau) p^m
\end{aligned} \]


Isolating the terms that are degree $k$ in $p$ on the first line yields a polynomial of degree $k$ in $Z$, with leading coefficient $-1/k$.
\end{proof}

For the Fricke-compatible case, we used the analogous result to give a strong genus zero statement, in Proposition 6.3 of \cite{C10}.  This was possible because the function $Z(g,h,\tau)$ had a pole at infinity.  Here, the question is more challenging, because the functions arising from non-Fricke compatible algebras are regular at infinity.

\subsection{Twisted modules - hypotheses}

We will use a little theory from vertex operator algebras and their twisted modules, but not in much depth.  As we remarked earlier, twisted modules were introduced in \cite{FLM88}, and were essential in the construction of the monster vertex algebra $V^\natural$.  While twisted modules have much subtle structure, we will only work with them explicitly in their capacity as vector spaces with a semisimple action of an operator $L_0$, whose eigenvalues produce a grading.  We will describe actions of groups that respect some essential additional structure, namely an internal action of the Virasoro algebra (in which $L_0$ is a distinguished element acting semisimply), and such actions will be called conformal.  Recall from before that the main player here is the monster vertex algebra $V^\natural$ endowed with a conformal action of the monster simple group $\MM$ with graded dimension $J(\tau)$ (as constructed and proved in \cite{FLM88}), and whose graded characters $T_g(\tau) := Tr(gq^{L_0-1}|V)$ are the Hauptmoduln listed by Conway and Norton \cite{CN79} (as proved by Borcherds \cite{B92}).

Let $G = \langle g \rangle \cong \Z/N\Z$ be a cyclic group with a conformal action on $V^\natural$, such that $f_{(1)}(\tau) := T_g(\tau) = Tr(gq^{L_0}|V^\natural)$ is invariant under $\Gamma_0(N)$.  For each $m|N$, we let $f_{(m)}(\tau) := T_{g^m}(\tau)$, which is invariant under $\Gamma_0(N/m)$.

From \cite{DLM00}, we have the following results:
\begin{enumerate}
\item (Proposition 12.4 and Theorem 10.3) For each integer $i$, there is a unique isomorphism class of irreducible $g^i$-twisted module of $V^\natural$.  We write it $V^\natural(g^i)$.
\item (Theorems 8.1 and 5.4) Each twisted module $V^\natural(g^i)$ admits a conformal action of $G$ compatible with the action of $G$ on $V^\natural$ (i.e., producing a $G$-equivariant module structure), such that $g^i$ acts as $e(L_0)$.
\item $\mathrm{Tr}(g^j q^{L_0-1}| V^\natural(g^i)) = C(i,j) \hat{F}_{i,j}(\tau)$ for nonzero constants $C(i,j)$.
\end{enumerate}

\noindent\textbf{Remark:} We wish to elaborate on the second result in the above list.  For each $g \in \MM$, the monster acts weakly on the category of $g'$-twisted modules, where $g'$ ranges over conjugates of $g$, by fixing the underlying vector space and conjugating the action of $V^\natural$.  This action is described in \cite{DLM00}, and works in the following way: the action of $V^\natural$ on a twisted module $V^\natural(g)$ is a map $m: V^\natural \otimes V^\natural(g) \to V^\natural(g)((z^{1/|g|}))$ satisfying some additional relations like associativity and a monodromy condition.  For any element $h \in \MM$, we form an $hgh^{-1}$ twisted module structure on $V^\natural(g)$ by $m^h(v \otimes a) := m(h\cdot v \otimes a)$.  By the uniqueness theorem (first result above), the stabilizer of the isomorphism class of an irreducible $g$-twisted module $V^\natural(g)$ is the centralizer of $g$.  By Schur's Lemma, there is therefore a canonical projective action of $C_\MM(g)$, the centralizer of $g$ in the monster, on $V^\natural(g)$, and this lifts to an action of a central extension $\C^\times . C_\MM(g)$, where $\C^\times$ acts by scalars.  In particular, each element of $C_\MM(g)$ admits a nonunique lift to an automorphism of each $V^\natural(g^i)$, but the product of two lifts of elements differs from the lift of the product of those elements by a nonzero scalar factor.  Since there are finitely many elements of $C_\MM(g)$, we may choose these scalars to be finite-order roots of unity, and the action of $\C^\times . C_\MM(g)$ descends to an action of a central extension $\widetilde{C_\MM(g)}$ of finite order.  If we restrict our attention to the cyclic group generated by $g$, we can lift to an action of a finite cyclic group that surjects to $\langle g \rangle$.  The second result of \cite{DLM00} is in fact that any choice of lift of $g^i$ on $V^\natural(g^i)$ is a scalar multiple of $e(L_0)$, but by virtue of the existence of the function $\hat{F}$, we may choose this lift to be given by an action of $G$, with the scalar multiple equal to one.  The action of $G$ on $V^\natural(g^i)$ for $(i,N) > 1$ is not uniquely defined by the $e(L_0)$ condition, but can be specified canonically by adding fusion information (to be treated later in this series).

We introduce some hypotheses on the structure of twisted modules.

\noindent\textbf{Hypothesis $A_g$:} $Tr(q^{L_0-1}| V^\natural(g^i)) = T_{g^i}(-1/\tau)$ for all $i$.

\noindent\textbf{Hypothesis $B_g$:} $C(i,j) = 1$ for all $i,j$.

Hypothesis $B_g$ is a refined form of the generalized moonshine conjecture restricted to the cyclic group $\langle g \rangle$.  Naturally, Hypothesis $A_g$ is the $j=0$ case of Hypothesis $B_g$.

\begin{lem} 
If $N$ is prime, then Hypothesis $A_g$ implies Hypothesis $B_g$.
\end{lem}
\begin{proof}
From result number 2 listed above, the action of $G$ on $V(g^i)$ is given by $g^i$ acting by multiplication by $e(L_0)$.  If $i \equiv 0$ (mod $N$), this follows from the moonshine conjecture: $V(g^i) = V^\natural$, and $C(i,j) = \frac{\mathrm{Tr}(g^j q^{L_0 - 1}| V^\natural)}{\hat{F}_{0,j}(\tau)} = 1$.  Otherwise, Hypothesis $A_g$ implies $Tr(q^{L_0-1}| V^\natural(g^i)) = \hat{F}_{i,0}$.  Since $N$ is prime, there is a unique $k$ such that $ik \equiv j$ mod $N$.  We write $V^\natural(g^i)_n$ for the subspace of $V^\natural(g^i)$ on which $L_0$ acts by multiplication by $n$.
\[ \begin{aligned}
\hat{F}_{i,j}(\tau) &= \hat{F}_{i,0}(\tau + k) \\
& = \sum_{n \in \frac{1}{N}\Z} \dim V^\natural(g^i)_{n+1} e(n(\tau + k)) \\
& = \sum_{n \in \frac{1}{N}\Z} e((n+1)k) \mathrm{Tr}(1|V^\natural(g^i)_{n+1}) e(n\tau) \\
&= \sum_{n \in \frac{1}{N}\Z} \mathrm{Tr}(g^{ik} | V^\natural(g^i)_{n+1} q^n \\
&= \mathrm{Tr}(g^j q^{L_0-1} | V^\natural(g^i))
\end{aligned} \]
\end{proof}

When $N$ is composite, we need a more detailed analysis of the action of $g$ on $V^\natural(g^i)$ for $(i,N)>1$ to remove root-of-unity ambiguities.

\begin{lem} 
Let $\mathcal{V}$ be the set of $G$-modules given by setting $V^{i,j/N}_k$ to be the subspace of $V^\natural(g^i)$ on which $g$ acts by $e(j/N)$ and $L_0$ acts by $k$.  If Hypothesis $B_g$ holds, then $E_g$ (as defined in Proposition \ref{prop:compatible}) is Fricke compatible with $(G,g,\mathcal{V})$ if $g$ is Fricke, and non-Fricke compatible with $(G,g,\mathcal{V})$ if $g$ is non-Fricke.
\end{lem}
\begin{proof}
By hypothesis $B_g$, for any nonnegative integers $i,j$, $\hat{F}_{i,j}$ is the trace of $g^j q^{L_0-1}$ on $V^\natural(g^i)$.  Taking the Fourier transform, we find that $c_{i,j}(ij/N)$ is equal to the dimension of both the $(i,j/N)$-root space of the Lie algebra $W_g$, and $V^{i,j/N}_{1+ij/N}$.  The compatibility then follows from Proposition \ref{prop:compatible}.
\end{proof}

\noindent\textbf{Hypothesis $C_g$:}  Let $\mathcal{V}$ be the set of $\widetilde{C_\MM(g)}$-modules given by setting $V^{i,j/N}_k$ to be the subspace of $V^\natural(g^i)$ on which $g$ acts by $e(j/N)$ and $L_0$ acts by $k$.  There is an action of $\widetilde{C_\MM(g)}$ on $W_g$ by homogeneous Lie algebra automorphisms, such that the Lie subalgebra $E_g$ of positive roots is Fricke-compatible with $(\widetilde{C_\MM(g)},g,\mathcal{V})$ if $g$ is Fricke, and non-Fricke compatible with $(\widetilde{C_\MM(g)},g,\mathcal{V})$ if $g$ is non-Fricke.

\noindent\textbf{Remark:} Because $G \subset \widetilde{C_\MM(g)}$, Hypothesis $C_g$ clearly implies Hypothesis $B_g$ (and hence $A_g$).  I believe I have a proof that for all $g \in \MM$, Hypotheses $A_g$ is true, and Hypothesis $B_g$ implies Hypothesis $C_g$.  If everything works out, the proof should appear later in this series.

\noindent\textbf{Remark:} Borcherds has pointed out that there is an alternative way to endow these Lie algebras with group actions, namely by choosing an explicit twisted denominator identity for each conjugacy class in $\widetilde{C_\MM(g)}$, in a way that is compatible with the series arising as characters of $\widetilde{C_\MM(g)}$ on the simple roots.  One needs only make this work for the first several simple roots, and then the compatibility follows from the Hecke-monic recursion.  Assuming one knows the character tables of $\widetilde{C_\MM(g)}$ for all $g \in \MM$ and all genus zero functions that are expected to appear, this can be done by computational brute force.  This strategy would yield the generalized moonshine conjecture without needing to appeal to twisted modules.  Unfortunately, no one seems to know either of the necessary pieces of information at the time of writing.

\subsection{Conditional generalized moonshine}

In this section, we choose a homomorphism $\pi: \Z/N\Z \to \MM$ sending a generator $g$ to an element of $\MM$, such that $\pi$ lifts to a homomorphism to a central extension $\widetilde{C_\MM(\pi(g))}$ of the centralizer of the image, that acts on all irreducible $g^i$-twisted modules of $V^\natural$.  In particular, this implies $T_g(\tau)$ is invariant under $\Gamma_0(N)$.

\begin{lem} \label{lem:CgHecke} 
Let $G = \widetilde{C_\MM(g)}$, and assume hypothesis $C_g$ is true.  Then $E_g$ is Fricke-compatible with $(G,g,\mathcal{V})$ if $g$ is Fricke, and non-Fricke compatible with $(G,g,\mathcal{V})$ if $g$ is non-Fricke.  Furthermore, for any $h \in G$, $Z$ is weakly Hecke-monic for $(g,h)$.
\end{lem}
\begin{proof}
The compatibility follows from Proposition \ref{prop:compatible}.  The weakly Hecke-monic property follows from Proposition \ref{prop:monic}.
\end{proof}

By Proposition 6.3 in \cite{C10}, this implies $Z(g,h,\tau)$ is a congruence genus zero function if $g$ is Fricke.  This is not sufficient to prove the generalized moonshine conjecture with our current state of knowledge, since we do not have such control for non-Fricke $g$, and we do not have the $SL_2(\Z)$-compatibility result.  Dong, Li, and Mason give a weak form of compatibility in the following sense: Let $V$ be a $C_2$ cofinite vertex operator algebra, let $(g,h)$ be a commuting pair of automorphisms with $|g| = T, |h| = T_1$, and let $M(T,T_1)$ be the commutative ring of holomorphic modular forms (non-negative weight, regular at cusps) for the group $\{ \binom{ab}{cd} \in SL_2(\Z) | b \in T\Z, c \in T_1\Z \}$.  There is then a space $C_1(g,h)$ of functions $S: (V \otimes M(T,T_1)) \times \HH$ that contains traces of twisted modules and admits ``slash'' operations, such that for any $S \in C_1(g,h)$ and $\binom{ab}{cd} \in SL_2(\Z)$, $S|_{\binom{ab}{cd}} \in C_1(g^a h^c, g^bh^d)$.  If $V$ is holomorphic and $g$-rational, then Theorem 10.1 of \cite{DLM00} implies $C_1(g,h)$ is spanned by these traces on twisted modules, and this yields a compatibility up to a constant, as predicted in the generalized moonshine conjecture.  Unfortunately, $g$-rationality seems to be quite difficult to prove in general.

We make the following conjecture for any $C_2$-cofinite vertex operator algebra $V$ with no negative weights, with $g$ and $h$ commuting finite order conformal automorphisms.  

\noindent\textbf{Hypothesis D:}  The space $C_1(g,h)$ (defined in section 5 of \cite{DLM00}) is spanned by a finite set of pseudo-traces (defined in section 4.1 of \cite{M04}) of lifts of $h$ acting on $g$-twisted Verma modules.  Furthermore, the subspace of functions with no logarithmic terms (i.e., powers of $\tau$) is spanned by traces of lifts of $h$ acting on irreducible $g$-twisted modules.

This hypothesis seems to be reasonable, in the sense that it is the ``least common generalization'' of Theorem 10.1 in \cite{DLM00} and theorem 5.5 in \cite{M04}.  We hope to treat this in later work - most of the proof in \cite{M04} transfers to the twisted setting, but there are significant technical hurdles.  In this paper, we are interested in the case $V = V^\natural$, where Theorem 10.3 in \cite{DLM00} implies the space of traces is one-dimensional.

\begin{thm} \label{thm:condmoon} 
Suppose hypothesis $C_g$ holds for all $g \in \MM$ and hypothesis $D$ holds for $V^\natural$ and all commuting pairs $(g,h)$.  Then the generalized moonshine conjecture is true (where $SL_2(\Z)$-compatibility is only given up to a nonzero constant).
\end{thm}
\begin{proof}
By Theorem 8.1 of \cite{DLM00}, $Z(g,h,\tau)$ is a specialization of an element of $C_1(g,h)$, so for any $\binom{ab}{cd} \in SL_2(\Z)$, $Z(g,h,\frac{a\tau+b}{c\tau+d})$ is a specialization of an element in $C_1(g^a h^c,g^b h^d)$.  From the finiteness assertion in Hypothesis D, there is an upper bound $M \geq 0$ on the degree of $\tau$ in expansions of $Z(g,h,\tau)$ at any cusp.

By Theorem 6.5 of \cite{DLM00}, the expansion of $Z(g,h,\frac{a\tau+b}{c\tau+d})$ has the form
\[ \sum_{i=0}^m \sum_{j = 0}^n \sum_{k=0}^\infty S_{i,j,k}q^{k+r_j} (2 \pi i \tau)^i. \]
We wish to show that $S_{i,j,k} = 0$ for $i>0$.  Suppose there are $j,k$ and $i>0$ such that $S_{i,j,k} \neq 0$, let $i_0$ be the largest value of $i$ such that this holds.  Let $j_0$ and $k_0$ be such that $S_{i_0,j_0,k_0}$ is nonzero, and $k_0 + r_{j_0}$ is minimal with respect to this property.  By Lemma \ref{lem:CgHecke}, 
\[ pT_pZ(g,h,\frac{a\tau+b}{c\tau+d}) = Z(g,h^p,\frac{pa\tau+pb}{c\tau + d}) + \sum_{i=0}^{p-1} Z(g^p,g^{-i}h,\frac{a\tau+b}{pc\tau+pd}+i) 
\]
is a monic polynomial of degree $p$ in $Z(g,h,\frac{a\tau+b}{c\tau+d})$, and therefore has a nontrivial power of $\tau$ of order $pi_0$.  Each of the summands on the right side of the equation is a specialization of an element of some $C_1(g',h')$, possibly after an affine coordinate change.  Affine coordinate changes fix the order of $\tau$ in the expansion, so the sum on the right side therefore has $\tau$ of order bounded above by $M$.  Since $pi_0>M$ for sufficiently large $p$, this is a contradiction.

The expansion of $Z(g,h,\frac{a\tau + b}{c\tau+d})$ therefore has no logarithmic terms, so by Hypothesis D, it is a nonzero constant multiple $Z(g^ah^c,g^bh^d,\tau)$.  This proves the $SL_2(\Z)$-compatibility up to a constant.  If there exist coprime integers $a,c$ such that $g^ah^c$ is Fricke, then Lemma \ref{lem:CgHecke} combined with Proposition 6.3 of \cite{C10} implies $Z(g^ah^c, g^bh^d,\tau)$ is a genus zero function of finite level for any integers $b,d$ satisfying $ad-bc=1$.  In this case, the $SL_2(\Z)$-compatibility implies $Z(g,h,\tau)$ is also a genus zero function of finite level.  If the integers $a$ and $c$ don't exist, then $Z(g,h,\tau)$ is regular at all cusps, and is therefore constant.

It remains to show that $Z(g,h,\tau) = J(\tau)$ if and only if $g=h=1$.  The leftward implication is clear - $J$ is the character of $V^\natural$.  The rightward implication is a consequence of the fact that the irreducible $g$-twisted modules of $V^\natural$ have strictly positive degree for $g \neq 1$, combined with the fact that nontrivial McKay-Thompson series have level greater than one.
\end{proof}

\noindent\textbf{Remark:} We don't need the full force of Hypothesis D, since the only part of the pseudo-trace statement that we use is the assertion that the powers of $\tau$ in the expansion of elements of $C_1(g,h)$ are bounded above.

\noindent\textbf{Remark:} This proof produces a slightly stronger statement of generalized moonshine: the constant ambiguity in the $SL_2(\Z)$-compatibility can be made a root of unity for the non-constant functions, because weakly Hecke-monic functions have poles with unit norm residue.  Unfortunately, we do not have enough information to restrict the ambiguity to a 24th root of unity.

\section{Open problems}

\begin{enumerate}

\item Given a completely replicable modular function $f$, there is a canonical minimal cyclic group $G$ for which $f$ is the restriction of a principally normalized Hecke-monic function on $\M_{Ell}^G$.  We could call this function the ``Hecke-monic envelope'' of $f$, since is it unique with respect to these conditions.  Calculations of Queen and Norton have produced functions that show up as the characters $Z(g,h,\tau)$ in generalized moonshine, and many are $PSL_2^+(\Q)$-transformations of functions that are replicable but not completely replicable.  These are attached to Hecke-monic functions on $\M_{Ell}^G$ for $G$ an abelian group that is not necessarily cyclic and not necessarily uniquely determined.  Is there a characterization of the abelian groups that can be attached to either the replicable modular functions, or the functions that show up as $Z(g,h,\tau)$?

\item The McKay-Thompson series satisfy the property that their expansions at zero have non-negative coefficients, and this is essential for constructing Lie algebras, since simple roots have non-negative multiplicity.  Non-negativity also holds for many non-monstrous completely replicable functions, such as 9a and 49a (where functions are named following the conventions in \cite{ACMS92}), but others have negative coefficients and therefore cannot be used to generate Lie algebras.  One obvious problem is: classify products with non-negative integer exponents.  While the existence of Lie algebras does not distinguish monstrous from non-monstrous functions, there is a slightly stronger condition conjecturally satisfied by McKay-Thompson series (implied by hypothesis $B_g$), namely that the vector-valued form $F$ have non-negative integer coefficients. We have not checked this condition for non-monstrous functions, but it is a weak form of demanding existence of twisted modules, i.e., assuming a uniqueness conjecture in \cite{FLM88}, McKay-Thompson series are distinguished by the property that they are characters of a conformal automorphism of a holomorphic vertex operator algebra with central charge 24, whose identity character is given by $J$.  Some non-monstrous cases of non-negativity may come from automorphisms of other vertex operator algebras, such as those arising from fixed-point-free automorphisms of Niemeier lattices.

\item This is not a particularly new question, but it would be nice to have a functorial characterization of congruence genus zero quotients of $\HH$ via moduli of structured elliptic curves.  For example, elements in the conjugacy class pA in the monster produce McKay-Thompson series that are invariant under $\Gamma_0^+(p)$, and the resulting modular curve $X_0^+(p) = \HH/\Gamma_0^+(p)$ parametrizes unordered pairs of elliptic curves, with dual degree $p$ isogenies between them.  Ogg pointed out in 1975 that the genus zero property of $X_0^+(p)$ characterizes those primes dividing $|\MM|$.  Substantial progress has been made in \cite{DF09} using moduli of solid tori to characterize McKay-Thompson series, and giving a characterization of genus zero quotients in terms of Rademacher sums.  

\item When the exponents in a product formula are not in $\N$, we can no longer describe the identity as the denominator formula of a Lie algebras in vector spaces.  However, if we have a suitable homomorphism from the Grothendieck semiring of a braided tensor category $\mathcal{C}$ to the complex numbers, we can sometimes employ the denominator identity to describe a Lie algebra object in $\mathcal{C}$ by checking that a pulled-back identity still holds.  For example, the monstrous superalgebras constructed in \cite{B92} are exactly of this form, when the category is that of super vector spaces.  Are there interesting algebraic objects attached to more general products (such as those containing cyclotomic integers, giving rise to cyclic anyons), and do they arise in physical contexts?

\item The correspondence between vector-valued modular functions for lattices and infinite products for $O(2,2)$ is far from a bijection, and its precise nature is far from clear.  For example, when a McKay-Thompson series $T_g(\tau)$ has a regular cusp, then there are infinite product expansions of $T_g(\sigma) - T_g(\tau)$ involving such cusps that do not arise from lifting the vector-valued functions we introduced.  In fact, one can use the exponents in a product to reverse-engineer Hecke-monic functions that are not principally normalized.  Are there structures in conformal field theory that produce these residual products?  We don't seem to have a enough a priori information in the Lie superalgebras to reconstruct all of the operator product expansions.  

\item Borcherds products like the ones in this paper appear in the string theory literature in many places, but with several different physical interpretations that are not obviously equivalent.  One often sees connections between these algebras and BPS states, but there doesn't seem to be a BPS theory explicitly attached to $V^\natural$.  Products in signature $(2,2)$ are also mentioned in \cite{D99}, with an explanation that string theory partition functions associated to a torus depend not only on conformal class but on worldsheet volume and a two-form field.  This expands the $T$-duality group from $SL_2(\Z)$ to $SO_{2,2}(\Z)$, which is commensurable with the symmetry groups of the automorphic functions we constructed.  This is somewhat related to the previous question, in the sense that one might expect the space of physical states of the $V^\natural$ conformal field theory to be related to a residual Lie superalgebra by a duality.  There is yet another interpretation of these Lie algebras as second quantizations in string field theory, arising from multiple bosonic string states (see, e.g., \cite{DMVV97} for generalities and \cite{DF09} for discussion of our specific case).  The product formulas arising in permutation orbifold theory (e.g., in \cite{T08}) are very similar to those in this paper, but with minus signs on the coefficients, i.e., the products are reciprocals of what we see here.  The reciprocal characters correspond to replacing $\bigwedge E$ in the Chevalley resolution with the symmetric algebra, suggesting a role for Koszul duality, and both the exterior and symmetric algebras on $E$ (with a suitable shift) have canonical cocommutative DG coalgebra structures, but it's not clear if they tell us anything useful.  Are these pictures related to each other?

\item We can think of the Lie algebras in this paper as categorifications of their (twisted) denominator formulas.  Lurie has asked what it means to categorify again, possibly with a view toward 2-equivariant elliptic cohomology and Hopkins-Kuhn-Ravenel characters.  It's not clear what sort of algebraic gadget we should get, but it ought to involve as an intermediate step the whole collection of Lie algebras at once as a manifestation of the equivariant conformal field theory.  Kremnizer has suggested some quiver-theoretic constructions, but I have been unable to find natural actions of finite groups on the resulting structures.



\end{enumerate}

\end{document}